\newif\ificmlver
\icmlvertrue
\newcommand{\titlename}{Multi-agent Performative Prediction with Greedy Deployment and Consensus Seeking Agents}

\ificmlver
\documentclass{article}

\usepackage{natbib}
\usepackage{fullpage}
\usepackage{algorithm}
\usepackage{algorithmic}
\usepackage[utf8]{inputenc} % allow utf-8 input
\usepackage[T1]{fontenc}    % use 8-bit T1 fonts
\usepackage[colorlinks=true, citecolor = blue]{hyperref}       % hyperlinks
\usepackage{url}            % simple URL typesetting
\usepackage{booktabs}       % professional-quality tables
\usepackage{amsfonts}       % blackboard math symbols
\usepackage{nicefrac}       % compact symbols for 1/2, etc.
\usepackage{microtype}      % microtypography
\usepackage{xcolor}         % colors
\usepackage{graphicx}

\title{\titlename}
\author{Qiang Li, Chung-Yiu Yau, Hoi-To Wai\thanks{Q.~Li, C.-Y.~Yau and H.-T.~Wai are with the Department of Systems Engineering and Engineering Management, The Chinese University of Hong Kong, Hong Kong SAR of China. Emails: \texttt{\{liqiang, cyyau, htwai\}@se.cuhk.edu.hk}}}

\usepackage{amsmath, amssymb, bm, bbm, multicol, Shortcuts_OPT, wrapfig, enumitem}

\else
\documentclass{article}

\usepackage{natbib}
\usepackage{fullpage}

\usepackage[utf8]{inputenc} % allow utf-8 input
\usepackage[T1]{fontenc}    % use 8-bit T1 fonts
\usepackage{hyperref}       % hyperlinks
\usepackage{url}            % simple URL typesetting
\usepackage{booktabs}       % professional-quality tables
\usepackage{amsfonts}       % blackboard math symbols
\usepackage{nicefrac}       % compact symbols for 1/2, etc.
\usepackage{microtype}      % microtypography
\usepackage{graphicx}
\usepackage{dsfont}
\usepackage{vwcol}
\usepackage{changepage}

\usepackage{enumitem,url,amsmath,amssymb,amsthm}

\usepackage{mdframed}

\usepackage{shortcuts_OPT, enumitem}

\newlength\figH
\newlength\figW
\usepackage[utf8]{inputenc}
\usepackage{pgfplots}
\DeclareUnicodeCharacter{2212}{−}
\usepgfplotslibrary{groupplots,dateplot}
\usetikzlibrary{patterns,shapes.arrows}
\pgfplotsset{compat=newest}
\usepackage{psfrag,subfigure,float,hyperref,bbm,cleveref}
\hypersetup{
  colorlinks   = true, %Colours links instead of ugly boxes
  urlcolor     = blue, %Colour for external hyperlinks
  linkcolor    = blue, %Colour of internal links 
  citecolor    = blue   %Colour of citations
}

\makeatletter
\def\Exa@space@setup{%
  \Exa@preskip=5cm plus 2cm minus 2cm
  \Exa@postskip=\Exa@preskip % or whatever, if you don't want them to be equal
}
\makeatother

\title{\titlename} 
\date{\today}

\fi
\newtheoremstyle{exampstyle}
  {1.5\topsep} % Space above
  {1.5\topsep} % Space below
  {} % Body font
  {} % Indent amount
  {\bfseries} % Theorem head font
  {.} % Punctuation after theorem head
  {.5em} % Space after theorem head
  {} % Theorem head spec (can be left empty, meaning `normal')

\theoremstyle{exampstyle}

\newtheoremstyle{otherstyle}
  {1.5\topsep} % Space above
  {1.5\topsep} % Space below
  {\itshape} % Body font
  {} % Indent amount
  {\bfseries} % Theorem head font
  {.} % Punctuation after theorem head
  {.5em} % Space after theorem head
  {} % Theorem head spec (can be left empty, meaning `normal')

\theoremstyle{otherstyle}
\newtheorem{Example}{Example}

\newtheorem{Prop}{Proposition}

\newtheorem{Corollary}{Corollary}

\newtheorem{Assumption}{A\!\!}

\begin{document}

\maketitle

\begin{abstract}
% This paper is concerned with a multi-agent performative prediction ({\aname}) problem. 
We consider a scenario where multiple agents are learning a common decision vector from data which can be influenced by the agents' decisions. This leads to the problem of multi-agent performative prediction ({\aname}). In this paper, we formulate {\aname} as a decentralized optimization problem that minimizes a sum of loss functions, where each loss function is based on a distribution influenced by the local decision vector.
% The algorithm design models a situation where agents are \emph{agnostic} to the performative nature of users and apply standard decentralized SGD to the learning problem. 
% Our setting can be applied to decentralized strategic classification for  where . 
% We study a {\bname} algorithm, which features greedy deployment and consensus seeking agents.
% for the scenario
% where agents are \emph{agnostic} to the performative nature of users. 
% The agents are consensus seeking as they exchange local decision vector with neighbors. Furthermore, the local decision vector is deployed to the local set of users prior to reaching a stationary point and/or consensus with other agents. 
We first prove the necessary and sufficient condition for the {\aname} problem to admit a unique multi-agent performative stable (Multi-PS) solution. We show that enforcing consensus leads to a laxer condition for existence of Multi-PS solution with respect to the distributions' sensitivities, compared to the single agent case. Then, we study a decentralized extension to the greedy deployment scheme \citep{mendler2020}, called the {\bname} scheme. We show that {\bname} converges to the Multi-PS solution and analyze its non-asymptotic convergence rate. Numerical results 
% on synthetic and real data 
validate our analysis. 
% an objective function composed of samples from $n$ different populations. Each of the population will react  
\end{abstract}

\section{Introduction}\vspace{-.2cm}
Traditional learning/prediction problems are often formulated with the assumption that data follows a \emph{static} distribution, from which a decision vector is sought by solving a corresponding optimization problem. While this assumption holds for `stationary' tasks such as image classification, many real world tasks are dynamical, involving data that could be influenced by decisions. In the latter case, the predictions are said to be \emph{performative}. Example scenarios include when users are \emph{strategic} \citep{Hardt, dong2018strategic, kleinberg2020classifiers} such as in training an E-mail spam classifier, where users (including spammers) can adapt to the classifier's rules to evade detection. 

A common way to model \emph{performative prediction} is via incorporating decision-dependent distribution in formulating the learning problem \citep{quinonero2008dataset}.
Since the pioneering work by \citet{perdomo2020performative}, there is a growing literature in analyzing the {performative prediction} problem, which studied the convergence of learning algorithms to stable point of performative prediction \citep{mendler2020,drusvyatskiy2020stochastic,brown2020performative,li2021state,wood2021online}, algorithms to find stationary solution of performative risk \citep{izzo2021learn,izzo2022learn,miller2021,ray2022decision}, update timescales of agent \citep{zrnic2021leads}, etc. 

Most of the existing works are focused on the single agent (learner) setting. In this paper, we concentrate on the multi-agent performative prediction ({\aname}) problem. Here, the agents are \emph{consensus-seeking} who seek a common decision vector that minimizes the sum of loss functions through communications on a graph/network. Such setup arises naturally where each agent acquires data from different subset/population of users, and they desire to seek a common decision vector for maximal generalization performance. For the example of training an E-mail spam classifier, each agent represents a regional server providing services to a subset of users. These users are in general different and may react to their serving agents differently, leading to heterogeneous and decision dependent data distributions. 
Notably, our setup differs from the recent works in \citep{narang2022multiplayer,piliouras2022multi} which consider a game theoretical setup of {\aname}; see Fig.~\ref{fig:multipfd} and \S\ref{sec:setup}. 

For the algorithmic model, we study a decentralized extension of the greedy deployment scheme \citep{mendler2020}, particularly we concentrate on the decentralized stochastic gradient (DSGD)  \citep{lian2017decentralized} algorithm with greedy deployment ({\bname}). Overall, the {\bname} scheme emulates a scenario where agents apply DSGD, a standard decentralized learning algorithm, while being \emph{agnostic} to the performative nature of the prediction problem. In particular, upon each iteration, the non-converged local decisions will be deployed directly. Then, agents optimize their decision using stochastic gradient constructed from data with distribution shifted by local decision. 
% During the algorithm, the agents communicate with nearby agents. Meanwhile, the intermediate decisions, that may not be in consensus nor stationary, will be directly deployed and observed by the users. 

Tackling {\aname} with the {\bname} scheme yields a practical scenario of performative prediction in the multi-agent setting. We inquire the open questions: \emph{When will the {\aname} problem admit a stable and consensual solution? If so, how fast does it take for {\bname} to converge to such solution? Does the limited communication in {\bname} impair its convergence rate?}
We provide affirmative answers to the above. Our idea hinges on adopting the concept of \emph{performative stable} solution \citep{perdomo2020performative}, which is a fixed point solution resulted from the interplay between agent(s) and data that react to the agent(s)' decisions. In particular, we focus on analyzing the multi-agent performative stable (Multi-PS) solution in our problem. 

To our best knowledge, this paper provides the first study and analysis of {\aname} with consensus seeking agents via a practical {\bname} scheme. We highlight the following key contributions:
\begin{itemize}[leftmargin=*, itemsep=0.1mm, topsep=.25mm]
\item We provide a \emph{necessary and sufficient} condition on the sensitivity of decision dependent data distributions for the existence and uniqueness of the Multi-PS solution. An interesting finding is that the Multi-PS solution exists even if some of the performative prediction problems of individual agents are unstable. The consensus seeking behavior in {\aname} tends to \emph{stabilize} the problem.
\item We study the {\bname} scheme and analyze its convergence towards the Multi-PS solution. We first show that the scheme is convergent under the same sufficient condition for existence of Multi-PS solution. With an appropriate step size rule, in expectation, the squared distance between the Multi-PS solution and the iterates decays as ${\cal O}(1/t)$, and the squared consensus distance decays as ${\cal O}(1/t^2)$, where $t$ is the iteration number. 
Our fine-grained analysis also reveals that heterogeneous users, poorly connected graph may slow down the convergence of {\bname}, but only so in the transient, and the asymptotic rate has a linear speedup property. 
% However, their effects will be overcome eventually as the iteration number grows.
\item To validate our analysis, we conduct numerical experiments on synthetic data and real data. Particularly, we consider the Gaussian mean estimation on synthetic data to validate our theory, and the logistics regression problem on the {\tt spambase} dataset. 
% and the strategic classification problem with logistics regression on . 
% The numerical results corroborate with our theoretical analysis.
\end{itemize}
The paper is organized as follows. \S\ref{sec:setup} introduces the {\aname} problem and {\bname} scheme. \S\ref{sec:main} presents the theoretical results on {\aname}, {\bname} together with a proof outline. Finally, \S\ref{sec:num} shows numerical experiments to validate our analysis. All proofs can be found in the appendix.\vspace{.05cm}
    % {\color{red} since we will have some space in this figure - perhaps we can illustrate the scenario used in  (and  as well.}} 
    
\textbf{Related Works.} As mentioned, the main ingredient of {\bname} scheme is the classical DSGD algorithm. The latter was introduced in \citep{sundhar2010distributed, bianchi2012convergence, sayed2014adaptation} and is recently popularized for decentralized learning. Of relevance to our work are \citep{lian2017decentralized} which demonstrated a {linear speed up} against centralized SGD, and \citep{pu2021sharp} with a refined analysis, also see \citep{kong2021consensus, yes_topology} on the effects of consensus steps and topology. We also mention recent advances to improve the efficiencies of DSGD in \citep{tang2018d, lan2020communication, koloskova2019decentralized}. Our analysis entails fresh challenges as the sought Multi-PS solution cannot be  described explicitly as an optimal solution to the {\aname} problem. 

Another line of relevant works pertains to (multi-agent) reinforcement learning whose formulation also entails a decision-dependent distribution. To this end, policy gradient algorithms are analyzed in \citep{zhang2020global, karimi2019non}, and they have been recently extended to the multi-agent setting \citep{zhang2018fully, chen2021communication}; also see the survey \citep{zhang2021multi}. Most of the above works showed convergence to a stationary point which may not be unique. This is in contrast to our analysis of {\bname} which converges to the unique Multi-PS solution.\vspace{-.3cm}

\section{Problem Setup}\vspace{-.3cm} \label{sec:setup}
Consider a scenario where there are $n$ agents connected on an undirected and connected graph $G = (V,E)$ such that $V = \{1,\ldots,n\}$, $E \subseteq V \times V$. Note that we include self-loops such that $(i,i) \in E$ for any $i \in V$. Each agent $i \in V$ draws samples from the $i$th population of users. The latter is characterized by the distribution ${\cal D}_i( \prm_i )$ supported on $\Zset \subseteq \RR^p$, $p \in \NN$ and parameterized by the agent's decision vector $\prm_i \in \RR^d$, $d \in \NN$. In other words, the $i$th population of users are only influenced by the $i$th agent's decision vector. Note that the populations of users can be \emph{heterogeneous} such that ${\cal D}_i( \prm ) \neq {\cal D}_j( \prm' )$, $i \neq j$, even if $\prm = \prm'$. 

The agents aim to find a \emph{common decision vector} $\prm \in \RR^d$ in a collaborative fashion that minimizes the average of local losses. Consider the \emph{multi-agent performative prediction (\aname) problem}:
\beq \label{eq:multipfd} \textstyle
\min_{ \prm_i \in \RR^d, \, i=1,\ldots,n }~\frac{1}{n} \sum_{i=1}^n \EE_{ Z_i \sim {\cal D}_i( \prm_i ) } \big[ \ell( \prm_i ; Z_i ) \big] ~~\text{s.t.}~~ \prm_i = \prm_j,~\forall~(i,j) \in E. 
\eeq 
Since $G$ is connected, the constraint $\prm_i = \prm_j$ for $(i,j) \in E$ enforces the decision to be in \emph{consensus} across the $n$ agents. 
In the above, $\ell( \prm_i ; Z_i )$ is the loss function of the fitness of the decision vector $\prm_i$ with respect to (w.r.t.) the sample $Z_i$. The expected value $\EE_{ Z_i \sim {\cal D}_i( \prm_i ) } [ \ell( \prm_i ; Z_i ) ]$ corresponds to the loss at agent $i$ w.r.t.~the samples from the $i$th population of users. 

\begin{Example}\label{example1}
We describe a strategic binary classification problem with linear utility for users to illustrate the application of \eqref{eq:multipfd}. The sample $Z_i$ is defined by a tuple $Z_i = ( {\bm X}_i, Y_i ) \in \RR^d \times \{ 0, 1 \}$ of feature and binary label, and the loss function is taken as the logistic regression function: 
\beq \textstyle \label{eq:log_loss}
\ell( \prm ; Z_i ) = \log\big( 1 + \exp( \pscal{ {\bm X}_i }{ \prm } ) \big) - Y \pscal{ {\bm X}_i }{ \prm } + \frac{\beta}{2} \| \prm \|^2,
\eeq 
where $\beta > 0$ is a regularization parameter.
The above loss function quantifies the mismatches between the classifier $\prm$ and the given data tuple $Z_i = ( {\bm X}_i, Y_i )$.

On the other hand, the distribution ${\cal D}_i(\prm_i)$ is controlled by the $i$th population of users. The distribution is \emph{decision dependent} such that the sample $Z_i \sim {\cal D}_i( \prm_i )$ depends on the $i$th decision $\prm_i$. Observing the $i$th decision $\prm_i$, users of the $i$th population provides samples that are modified via a linear utility such that $Z_i = ({\bm X}_i , Y_i) \sim {\cal D}_i (\prm_i)$ is given by
\beq \textstyle \label{eq:perf_quad}
{\bm X}_i = \argmax_{ \hat{\bm X} \in \RR^d } \big\{ \pscal{ \prm_i }{ \hat{\bm X} } - \frac{1}{2\epsilon_i} \| \hat{\bm X} - {\bm X} \|^2 \big\},~Y_i = Y~~\text{with}~~( {\bm X}, Y) \sim {\cal D}_i^{\sf o},
\eeq 
for some $\epsilon_i > 0$, 
where ${\cal D}_i^{\sf o}$ is a base data distribution of the $i$th population. Note that in the above, we have the closed form solution ${\bm X}_i = {\bm X} + \epsilon_i \prm_i$.
Tackling \eqref{eq:multipfd} leads to a common classifier $\Bprm$ which takes the effects of the heterogeneous decision dependent distributions into account. 

% Applications where \eqref{eq:log_loss}, \eqref{eq:perf_quad} apply include the training of an E-mail spam classifier through cooperation between multiple companies. Here, each agent is a company which serves different set of users. Note that the users are not identical and they may respond differently to the classifier deployed by their associated company. It is beneficial for the companies to collaborate in order to leverage the experiences observed by the others.
% \hfill $\square$
\end{Example}

The \aname~problem \eqref{eq:multipfd} comprises of a stochastic objective function with a decision dependent distribution. In particular, for each agent $i$, the distribution ${\cal D}_i( \prm_i )$ captures a feedback mechanism where users of the $i$th population react to the decision made by the $i$th agent. See Fig.~\ref{fig:multipfd} (left) for an illustration. Due to non-convexity, the performative risk in \eqref{eq:multipfd} can be difficult to minimize. 
% {\color{red} Shall we discuss about the solution concept, e.g., performative stable solution? (or we shall leave it later)}
% \aaa{Due to the coupling} nature between distribution ${\cal D}_{i}(\prm_i)$ and loss function $\ell(\prm_i; Z_i)$, 
In this paper, we are interested in the \emph{multi-agent performative stable} (Multi-PS) solution:
\beq  \label{eq:thps}
\textstyle \thps = {\cal M}( \thps) \eqdef \argmin_{\prm \in \RR^d} \frac{1}{n}\sum_{i=1}^{n}\EE_{ Z_i \sim {\cal D}_i(\thps)} [\ell(\prm ; Z_i )]
\eeq 
which approximately solves \eqref{eq:multipfd}. Notice that $\thps$ is defined to be a fixed point of the map ${\cal M}: \RR^d \to \RR^d$. The existence and uniqueness of $\thps$ will be shown under mild condition in Proposition~\ref{lem:exist}. 

\begin{figure}[t] 
    \centering
    \includegraphics[width=.31\linewidth]{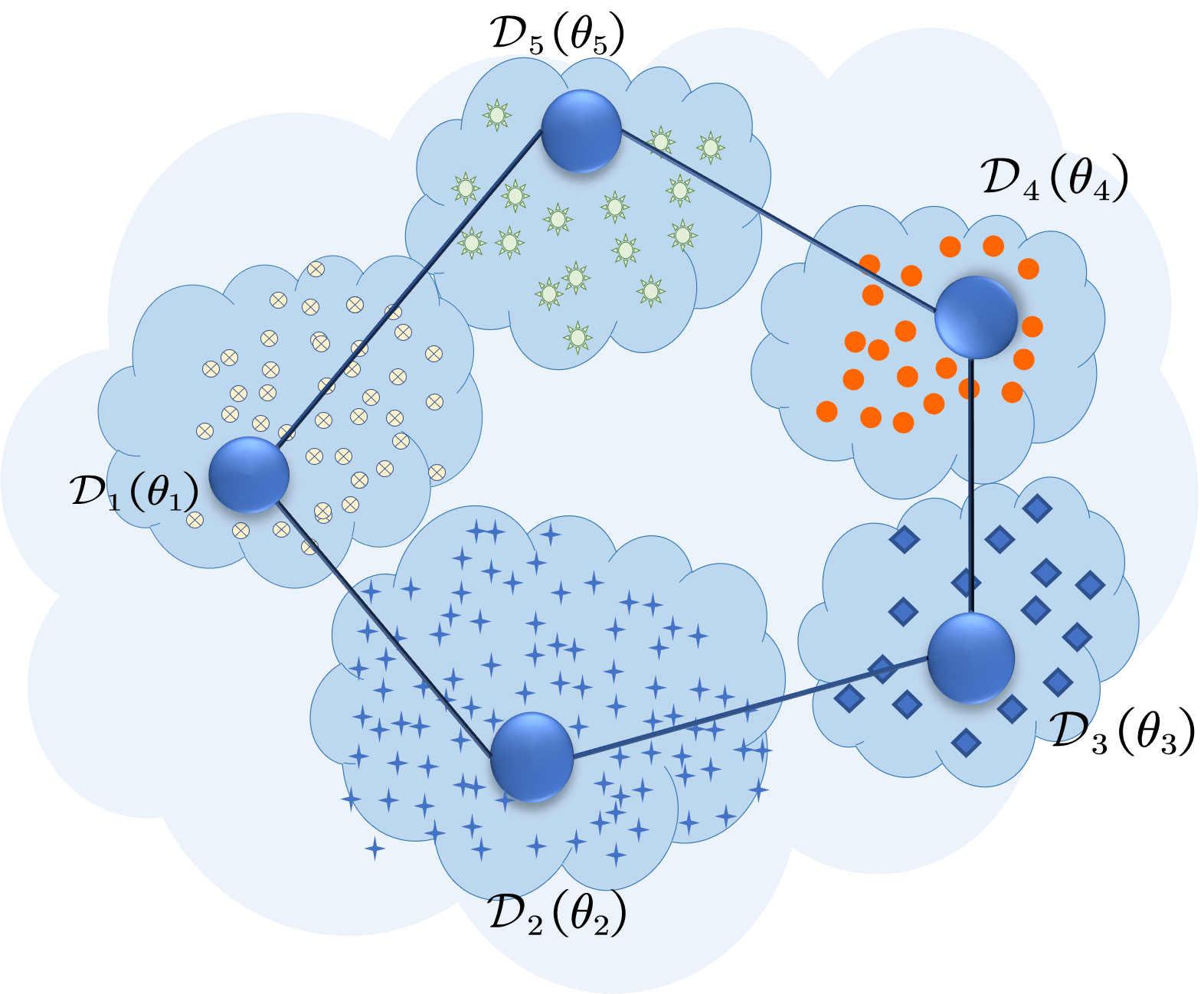}
    ~~
    \includegraphics[width=.3\linewidth]{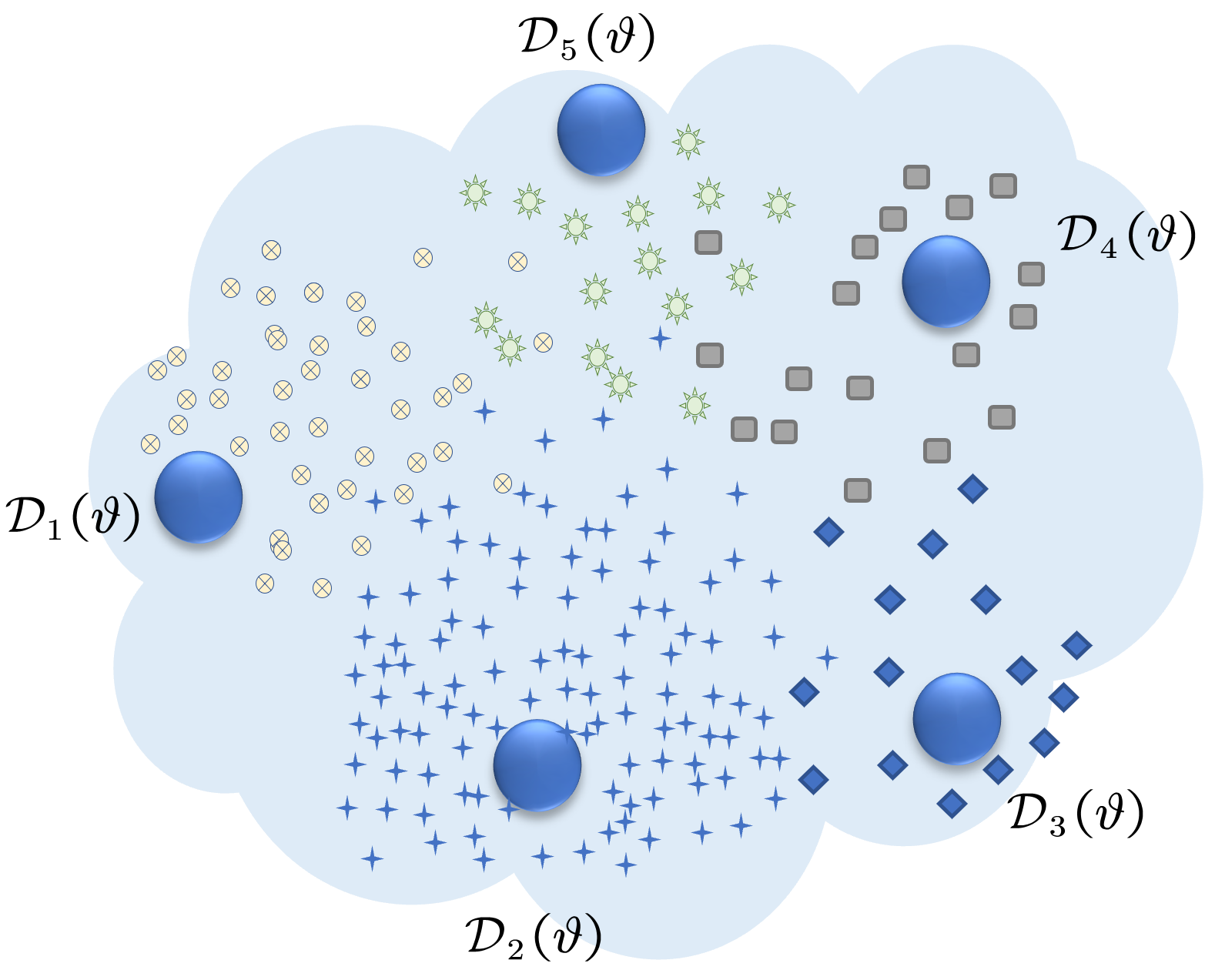}
    ~~
    \includegraphics[width=.3\linewidth]{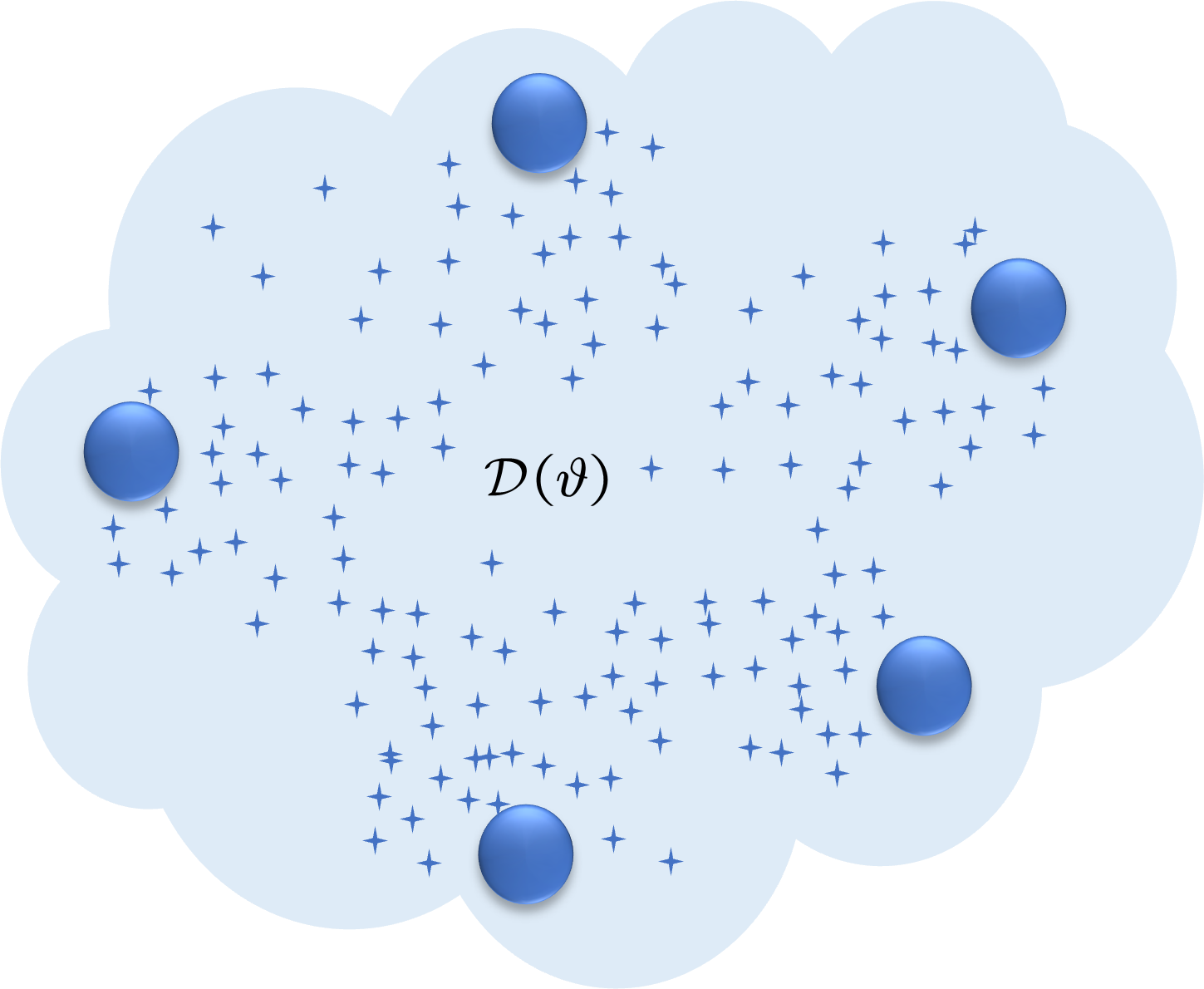}
    \caption{\textbf{Comparing the frameworks of multi-agent/player performative prediction}. (Left) {\color{blue} This work}: agent deploys $\prm_i$ and local distribution ${\cal D}_i( \prm_i )$ is affected by $\prm_i$ only; the agents communicate while aiming to reach consensus $\prm_1=\cdots=\prm_n$ eventually. (Middle) \citet{narang2022multiplayer}: agent deploys $\prm_i$ but local distribution is affected by joint decision ${\cal D}_i(\prm_1,\ldots,\prm_n)$. (Right) \citet{piliouras2022multi}: similar to {\protect\hypersetup{citecolor=.}\citep{narang2022multiplayer}} but the distribution is global ${\cal D}(\prm_1,\ldots,\prm_n)$.}\vspace{-.2cm} \label{fig:multipfd}
\end{figure}

\textbf{Comparison to Existing Works.} Our setup differs from recent works in \citep{piliouras2022multi,narang2022multiplayer}. Unlike ours, both works considered a game theoretical setting where the agents do not directly communicate their decisions. 
Instead, agents are coupled through the data distributions which depend on the global decision profile $\vartheta \eqdef (\prm_1, \ldots, \prm_n)$. 
% These frameworks are compared in Fig.~\ref{fig:multipfd}.

In \citep{narang2022multiplayer}, agent/player $i$ seeks to solve $\min_{ \prm_i } \EE_{Z_i\sim {\cal D}_{i}(\vartheta)} [ \ell_i(\vartheta; Z_i) ]$ where the local distribution ${\cal D}_i$ depends on the global decision $\vartheta$ [Fig.~\ref{fig:multipfd} (middle)]; while \citet{piliouras2022multi} study a slightly different setting where the goal of agent/player $i$ is to solve $\min_{ \prm_i } \EE_{Z_i\sim {\cal D}(\vartheta)} [ \ell(\prm_i; Z_i) ]$ [Fig.~\ref{fig:multipfd} (right)]. Both works showed convergence to a stable equilibria via (stochastic) gradient-type updates. 
In contrast, our work considers cooperative agents who are willing to share their decisions with other agents. Under this setting, we show that a natural decentralized scheme achieves convergence to a consensual decision that is performatively stable.\vspace{.1cm}

\textbf{Decentralized Scheme.} 
To find an approximate solution to \eqref{eq:multipfd} in a cooperative fashion, the natural approach is to deploy a decentralized scheme where agents communicate limited amount of information with neighbors that are directly connected on $G$. 
Furthermore, we focus on a scenario where agents are \emph{agnostic} to that their users may react to the agents' decisions. 
The above motivates us to study the following variant of the standard DSGD algorithm \citep{lian2017decentralized}. 

Let us define a mixing matrix ${\bm W} \in \RR_+^{n \times n}$ on $G$ such that $W_{ij} = W_{ji} > 0$ if $(i,j) \in E$, otherwise $W_{ij} = 0$ if $(i,j) \notin E$, satisfying $\sum_{i=1}^n W_{ij} = 1$ for any $i=1, \ldots, n$. We consider the scheme:

\begin{center}
\fbox{\begin{minipage}{.975\linewidth}
\begin{center}\underline{\textbf{DSGD with Greedy Deployment (\bname) Scheme}}\end{center}
At iteration $t = 0, 1, \ldots$, for any $i \in V$, agent $i$ updates his/her decision ($\prm_i^t$) by the recursion consisting of two phases
\beq \textstyle \label{eq:dsgd}
{\sf (Phase~1)}~~
Z_i^{t+1} \sim {\cal D}_i(\prm_i^t)
~~
{\sf \Big|}
~~
{\sf (Phase~2)}~~
\prm_i^{t+1} = \sum_{j=1}^n W_{ij} \prm_j^t - \gamma_{t+1} \grd \ell( \prm_i^t ; Z_i^{t+1} ),
\eeq 
where $\gamma_{t+1} > 0$ is a step size. 
Note that $\grd \ell( \prm_i^t ; Z_i^{t+1} )$ denotes the gradient taken w.r.t.~the first argument $\prm_i^t$, and the samples $Z_i^{t+1}$ at each agent are independent of each other.
\end{minipage}}
\end{center}

The above combines the DSGD algorithm \citep{lian2017decentralized} with greedy deployment \citep{mendler2020}. 
For iteration $t \geq 0$, the scheme can be described with two phases: 

In \textsf{Phase 1}, agent $i$ deploys $\prm_i^t$ at the $i$th population, whose users react to the decision and reveal a sample $Z_i^{t+1} \sim {\cal D}_i( \prm_i^t )$ to the agent. Note that this is a greedy deployment scheme since the deployed decision $\prm_i^t$ may not be stable w.r.t.~\eqref{eq:multipfd}. In \textsf{Phase 2}, agent $i$ receives the current decisions $\prm_j^t$ from his/her neighbors $j \in {\cal N}_i$ and update $\prm_i^{t+1}$ according to the `consensus' + `stochastic gradient' step. Here, the `stochastic gradient' step is based on the local decision $\prm_i^t$ and sample $Z_i^{t+1}$ taken in phase one with greedy deployment. We remark that throughout the {\bname} scheme, the agents remain unaware of the performative behavior of their users. 

Notice that the {\bname} scheme can be interpreted as a DSGD algorithm deploying \emph{biased} stochastic gradient updates. Denote $\EE_t[\cdot]$ as the conditional expectation up to iteration $t$, we observe 
\beq 
\EE_t [ \grd_{\prm_i^t} \ell( \prm_i^t ; Z_i^{t+1} ) ] = \EE_{Z_i \sim {\cal D}_i( \prm_i^t) } [ \grd_{\prm_i^t} \ell( \prm_i^t ; Z_i ) ] \neq \grd_{\prm_i^t} \big\{ \EE_{ Z_i \sim {\cal D}_i( \prm_i^t ) } [ \ell( \prm_i^t ; Z_i ) ] \big\}.
\eeq 
The unbiased gradient on the r.h.s.~also involves derivative w.r.t.~the decision in the distribution. Analyzing the convergence of {\bname} therefore requires different techniques.\vspace{-.2cm}

\section{Main Results}\vspace{-.2cm} \label{sec:main}
This section studies the convergence of the {\bname} scheme and demonstrates that the latter can approximately solve the {\aname} problem \eqref{eq:multipfd}.
To facilitate our discussions, we first define:
\beq \textstyle \label{eq:perfrisk}
f_i( \prm ; \Bprm ) := \EE_{ Z_i \sim {\cal D}_i( \Bprm) } [ \ell( \prm; Z_i ) ],~~f( \prm; \Bprm ) := \frac{1}{n} \sum_{i=1}^n f_i( \prm ; \Bprm ).
\eeq 
Note that the first arguments in $f_i, f$ denote the agent's decision and the second argument is the deployed decision known to the population of users. For the rest of this paper, unless otherwise specified, $\grd \ell( \prm; Z )$, $\grd f_i( \prm; \prm' ), \grd f( \prm; \prm' )$ denote the gradients taken w.r.t.~the first argument $\prm$. 

Using the above notations, \eqref{eq:multipfd} is equivalent to $\min_{\prm} f( \prm; \prm )$.
We consider the set of assumptions:
\begin{Assumption}\label{ass: strongcvx}
Fix any $\bar{\prm} \in \RR^d$, the function $f( \prm; \bar{\prm} )$ is $\mu$-strongly convex in $\prm$ such that
\beq
f( \prm^\prime ; \bar{\prm} ) \geq f( \prm ; \bar{\prm} ) + \pscal{ \grd f( \prm; \bar{\prm} ) }{ \prm^\prime - \prm } + (\mu/2) \| \prm^\prime - \prm \|^2,~\forall~\prm^\prime, \prm \in \RR^d.
\eeq
% where $\grd f( \prm; \bar{\prm} )$ denotest the gradient taken w.r.t.~the first argument $\prm$.
\end{Assumption}

\begin{Assumption}\label{ass: smooth}
For any $i = 1, \ldots, n$, the loss function $\ell(\prm;z)$ is $L$-smooth such that
\beq
\norm{\grd \ell(\prm; z)-\grd \ell(\prm^\prime; z^\prime)} \leq L\{\norm{\prm - \prm^\prime} + \norm{z-z^\prime}\},~\forall~\prm^\prime, \prm \in \RR^d,  z, z^\prime \in \Zset.
\eeq
\end{Assumption}

\begin{Assumption}\label{ass:sensitive}
For any $i=1,\ldots,n$, there exists a constant $\epsilon_i>0$ such that
\beq 
{\cal W}_{1}({\cal D}_i(\prm), {\cal D}_i(\prm^\prime)) \leq \epsilon_i \norm{\prm-\prm^\prime},~\forall~\prm^\prime, \prm \in \RR^d,
\eeq
where ${\cal W}_{1} ( {\cal D}, {\cal D}' )$ denotes the Wasserstein-1 distance between the distributions ${\cal D}, {\cal D}'$.
\end{Assumption}
A\ref{ass: strongcvx}, A\ref{ass: smooth} require the loss functions to be strongly convex and smooth, while A\ref{ass:sensitive} states that the amount of distribution shift caused by the reaction of $i$th population to the agent's decision grows linearly with difference in decision. These assumptions are standard in the literature, e.g., \citep{perdomo2020performative, mendler2020, drusvyatskiy2020stochastic}.
To simplify notations, we define the average, maximum sensitivity as $\epsilon_{\sf avg} := \sum_{i=1}^n \epsilon_i / n$, $\epsilon_{\sf max} := \max_{i=1,\ldots,n} \epsilon_i$, respectively. 

Our first result establishes the existence of the Multi-PS solution $\thps$ satisfying $\grd f( \thps; \thps ) = {\bm 0}$:
\begin{Prop}[\bf Existence and Uniqueness of $\thps$]\label{lem:exist} Under A\ref{ass: strongcvx}--A\ref{ass:sensitive}. Define the map ${\cal M}: \RR^d \to \RR^d$
\beq \label{eq:map_M} \textstyle 
    \mathcal{M}(\prm) = \argmin_{\prm^\prime \in \RR^d} \frac{1}{n} \sum_{i=1}^{n} f_{i}(\prm^\prime; \prm) 
\eeq
If $\epsilon_{\sf avg} < \mu / L$, then the map ${\cal M}(\prm)$ is a contraction with the unique fixed point $\thps = {\cal M}( \thps )$. 
If $\epsilon_{\sf avg} \geq \mu / L$, then there exists an instance of \eqref{eq:map_M} where $\lim_{T \to \infty} \norm{ {\cal M}^T(\prm) } = \infty$.
\end{Prop}
See \S\ref{app:exist} for the proof.
Note the inverse condition number $\mu/L$ yields a \emph{tight} threshold on the sensitivity of the population for the stability of {\aname}. 
Our result can be viewed as the multi-agent extension to \citep[Prop.~4.1]{perdomo2020performative}. 
Notice that while $\thps$ may not solve \eqref{eq:multipfd}, it yields an approximate solution to the latter, depending on the magnitude of $\epsilon_{\sf avg}$; see \citep{perdomo2020performative}. 
% The latter shows the existence of unique performative stable solution with single agent under $\epsilon_i < \mu / L$. 
 
Our result shows that {\aname} has a relaxed requirement for the existence of Multi-PS solution as it only depends on the \emph{average sensitivity} $\epsilon_{\sf avg}$. 
Consider when $\epsilon_i$ exceeds $\mu/L$ for the population served by $i$th agent, now there may not exist a performative stable solution for the \emph{individual} 
% performative prediction of the 
agent
% independently applying greedy deployment  by the agent may lead to a divergent scheme  
\citep{perdomo2020performative}. Meanwhile, by Proposition~\ref{lem:exist}, the \emph{Multi-PS} solution still exists as long as $\epsilon_{\sf avg} < \mu/L$, e.g., when there are agents in the network with less sensitive users. 

We also compare Proposition~\ref{lem:exist} to \citep{narang2022multiplayer}. Note that the latter considers decision-dependent distributions ${\cal D}_i ( \vartheta )$ with $\vartheta := ( \prm_1, \ldots, \prm_n )$ involving all decisions, which is different from our setting. Nevertheless, under similar conditions to A\ref{ass: strongcvx}--A\ref{ass:sensitive}, \citet[Theorem 1]{narang2022multiplayer} shows the performative stable equilibria exists if $\sum_{i=1}^n \epsilon_i^2 < \mu^2/L^2$. When $\epsilon_i = \epsilon_{\sf avg}$, it yields $\epsilon_{\sf avg} < \mu / (\sqrt{n}L)$ which is more restrictive than the requirement in Proposition~\ref{lem:exist}.  

The relaxed condition in Proposition~\ref{lem:exist} can be attributed to the consensus seeking nature of \eqref{eq:multipfd}, where agents with less sensitive users serve as mediators that prevent divergence of the map \eqref{eq:map_M}.\vspace{.1cm}
% \aaa{Under homogeneous data case, $\epsilon_i = \epsilon_j =\epsilon$, $\forall i, j\in V$. Then, Proposition \ref{lem:exist} means that if $\epsilon<\mu/L$, ${\cal M}(\prm)$ is a contraction map. In \citep{narang2022multiplayer}, the convergence region of Repeated stochastic gradient method which is developed to find the performative stable point under non-consensus scenario is $\sum_{i=1}^{n}\epsilon_i^2 \leq \frac{\mu^2}{L^2}$. Under homogeneous data case, their convergence region reduce to $\epsilon<{\mu}/{(\sqrt{n}L)}$, which shows that consensus step can enlarge the converge region. Besides, for single agent case, both our paper and \citep{narang2022multiplayer} coincide with same convergence region. }

\textbf{Convergence Analysis of {\bname} Scheme.} Next, we focus on analyzing the convergence of the {\bname} scheme. We require the following assumptions:
\begin{Assumption}\label{ass: graph}
The non-negative matrix ${\bm W}$ is doubly stochastic, i.e., ${\bm W}\mathbf{1}={\bm W}^\top\mathbf{1}=\mathbf{1}$. There exists a constant ${\rho} \in(0,1]$ 
% and a projection matrix ${\bm U}$ 
such that 
% $\boldsymbol{I}-\frac{1}{n} \mathbf{1} \mathbf{1}^{\top}={\bm U} {\bm U}^{\top}$ and 
$\left\| \bm{W} - (1/n) {\bf 1}{\bf 1}^\top \right\|_{2} \leq 1-{\rho}$.
\end{Assumption}
\begin{Assumption}\label{ass:SecOrdMom}
For any $i=1,\ldots, n$ and fixed $\prm \in \RR^d$, 
% it holds $\EE_{z_i \sim {\cal D}_i(\prm )} [ \grd \ell( \prm ; z_i ) ] = \grd f_i (\prm)$ and 
there exists $\sigma \geq 0$ such that\vspace{-.05cm}
\begin{align}
\EE_{Z_i \sim {\cal D}_i( \prm )} [ \| \grd \ell( \prm ; Z_i ) - \grd f_i( \prm; \prm ) \|^2 ] \leq \sigmatwo^2 (1+\norm{ \prm- \thps }^2) . \label{A4-1}
\end{align}
\end{Assumption}
A\ref{ass: graph} is a standard assumption on the mixing matrix, for example if $G$ is a connected graph, then ${\bm W}$ satisfying the condition can be constructed \citep{boyd2004fastest}. Meanwhile, A\ref{ass:SecOrdMom} bounds the variance of the `stochastic gradient' $\grd \ell( \prm; Z_i )$ as the expected value to the latter yields $\EE_{Z_i \sim {\cal D}_i( \prm )} [ \grd \ell( \prm; Z_i ) ] = \grd f( \prm ; \prm  )$, where the gradients are taken w.r.t.~the first argument $\prm$.
Moreover, we assume that:
\begin{Assumption}\label{ass:hete}
For any $i=1,\ldots,n$, there exists $\varsigma \geq 0$ such that \vspace{-.05cm}
\begin{align}\label{eq:hete_grad} 
    \| \grd f(\prm; \prm) - \grd f_i( \prm; \prm) \|^2 \leq \varsigma^2(1+\norm{\prm-\thps}^2),~\forall~\prm \in \RR^d.
\end{align}
% where the gradients are taken w.r.t.~the first argument only.
\end{Assumption}
It bounds the \emph{heterogeneity} of the locally observed samples. As $\grd f_i( \prm; \prm) = \EE_{ Z_i \sim {\cal D}_i( \prm ) } [ \grd \ell (\prm; Z_i )]$, if ${\cal D}_i( \prm ) = {\cal D}_j (\prm)$, the constant will be $\varsigma = 0$ when the population associated with each agent produces identically distributed samples with the same deployed decision vector. Otherwise, $\varsigma > 0$ measures degree of heterogeneity across populations. We remark that even when ${\cal D}_i( \prm ) = {\cal D}_j (\prm)$, the samples $Z_i^{t+1}, Z_j^{t+1}$ may still not be identically distributed \emph{during the {\bname} iterations} since the decision vectors $\prm_i^t, \prm_j^t$ may not be in consensus when $t < \infty$.

A\ref{ass:hete} also implies $\max_{i=1,\ldots,n} \normtxt{ \grd f_i ( \thps; \thps ) }^2 \leq \varsigma^2$. In fact, as we show in \eqref{eq:cons_noA6} of the appendix, it suffices to prove the convergence of {\bname} without A\ref{ass:hete} as in \citep{pu2021sharp, yuan2021removing}.
% {\color{red}(+ some other works?)}. 
We proceed our analysis with A\ref{ass:hete} to extract a tighter bound especially when $\varsigma$ is small. 

In both A\ref{ass:SecOrdMom}, A\ref{ass:hete}, we allowed the upper bounds to grow with ${\cal O}(1 + \| \prm - \thps \|^2)$, e.g. A\ref{ass:SecOrdMom} is similar to \citep[Assumption 1]{pu2021sharp}. This is  a weaker condition than the commonly assumed uniform bound of ${\cal O}(1)$, e.g., in \citep{lian2017decentralized}. For example, it covers situations when the quadratic components in $\ell(\cdot)$ depends on both $Z$ and $\prm$. Note that this relaxation implies that the variance/heterogeneity can be unbounded, leading to additional challenges in our analysis. 

Define the following quantities and constants: fix any $\delta > 0$,
\begin{align}
& \textstyle \Bprm^t := (1/n)\sum_{i=1}^n \prm_i^t,~~\Tprm^t := \Bprm^t - \thps,~~\CSE{t} := \big( \prm_1^t~\cdots~\prm_n^t \big) - \Bprm^t {\bf 1}^\top,~~\tmu \eqdef \mu - (1+\delta)\epsilon_{\sf avg} L, \notag \\[.1cm]
& \displaystyle c_1 := \frac{L(1+\epsilon_{\sf max})^2}{2n\delta \epsilon_{\sf avg}},~~c_2 \eqdef 4 \left( \frac{ \sigma^2 }{n} + L^2 ( 1 + \epsilon_{\sf max} )^2 \right), ~~c_3\eqdef 12 \sigma^2 + 18L^2(1+ \epsilon_{\sf max})^2. \label{eq:c1c2c3}
\end{align} 
Note that $\Tprm^t$ is the distance between the $t$th averaged iterate $\Bprm^t$ and $\thps$, while $\CSE{t}$ is a $d \times n$ matrix of the $t$th \emph{consensus error}.
The following theorem establishes the convergence rate of {\bname}:
\begin{center}
\fbox{\begin{minipage}{.975\linewidth}
\begin{theorem}\label{thm1}
Under A\ref{ass: strongcvx}--A\ref{ass:hete} and the condition on average sensitivity that $\epsilon_{\sf avg} < \frac{\mu}{(1+\delta) L}$. Suppose the step sizes $\{ \gamma_{t} \}_{t \geq 1}$ satisfy $\gamma_{t+1} \leq \gamma_{t}$ for any $t \geq 1$, 
\beq \label{eq:stepsizecond}
\sup_{t \geq 1}~\gamma_{t} \leq \min\left\{ \frac{4}{\tmu},~ \frac{ \tmu }{ c_2 } , \frac{ \rho }{ \sqrt{2 c_3} }, 
\sqrt{ \frac{ \rho^2 \tmu }{ 192 c_1 ( \sigma^2 + \varsigma^2 ) } } , \frac{ \rho c_1 }{ 4 \tmu c_1 + \rho c_2 }
\right\},
\eeq 
and the condition $\frac{ \gamma_{t} }{ \gamma_{t+1} } \leq \min\{ \sqrt{1 + (\tmu/4) \gamma_{t+1}^2}, \sqrt[3]{ 1 + (\tmu/4) \gamma_{t+1}^3 }, 1 + \rho/(4-2\rho) \}$ for any $t \geq 1$. 
% \beq 
%     \sup_{t\geq 1}~\gamma_{t+1} \leq \min\left\{\frac{\rho}{2}\sqrt{\frac{\tmu}{ c_1 (c_3+9\varsigma^2)}}, ~~\frac{\tmu}{c_2},~~ \frac{c_1}{2c_2}, ~~ \frac{\rho}{\max\{8\tmu, \sqrt{2 c_3}\}}
%     \right\}.
% \eeq
Then, the iterates generated by {\bname} admit the following bound for any $t \geq 0$,
\begin{align}
& \EE \left[ \norm{ \Tprm^{t+1} }^2 
% \! + \! \gamma_{t+1} \frac{4c_1}{n \rho} \norm{\CSE{t+1}}_F^2
\right] \leq 
\prod_{i=1}^{t+1} \left( 1 - \frac{\tmu \gamma_i}{ 2 } \right) {\sf D} + \frac{ 288 c_1 ( \sigma^2 + \varsigma^2 ) }{ \rho^2 \tmu } \gamma_{t+1}^2 + \frac{ 8 \sigma^2 }{ \tmu n } \gamma_{t+1},
% \prod_{i=1}^{t+1} \left( 1 - \tmu \gamma_{i}\right) {\sf D} + \frac{72 c_1 (\sigma^2 + \varsigma^2)}{n\tmu \rho^2} \gamma_{t+1}^2 +  \frac{4\sigma^2}{n \tmu}\gamma_{t+1},  
\label{eq:mainthm} \\
& \EE \left[ {\textstyle \frac{1}{n}} \norm{\CSE{t+1}}^2_F \right] \leq \left(1- \frac{\rho}{2} \right)^{t+1} {\textstyle \frac{1}{n}} \norm{\CSE{0}}_F^2 + \frac{ 2(9 + 12 \overline{\Delta}) (\sigma^2+\varsigma^2) }{\rho^2} \gamma_{t+1}^2, \label{eq:cseerror}
\end{align}
where ${\sf D}\eqdef \| \Tprm^0 \|^2 + \gamma_{1} \frac{8c_1}{n \rho} \norm{\CSE{0}}^2$ denotes the initial error, $\overline{\Delta} := {\sf D} + \frac{3}{2} + \frac{8 \sigma^2}{ c_2 n }$, and we recall the definitions of the constants $c_1, c_2, c_3$ from \eqref{eq:c1c2c3}.
\end{theorem}
\end{minipage}} \end{center}

The free parameter $\delta > 0$ in \eqref{eq:c1c2c3} can be chosen arbitrarily. Thus, according to Theorem~\ref{thm1}, {\bname} converges when the average sensitivity $\epsilon_{\sf avg}$ is strictly below the threshold $\mu/L$ in Proposition~\ref{lem:exist}, i.e., the same sufficient and necessary condition that guarantees the existence of $\thps$. 
Further, our result holds for general step size rules such as constant step size and diminishing step size. 

To yield $\EE[ \| \Tprm^t \|^2 ] \to 0$, a common option is $\gamma_t = \frac{ a_0 }{ a_1 + t}$ for some $a_0, a_1 > 0$. In the latter case as $\gamma_t = {\cal O}(1/t)$, we observe from \eqref{eq:mainthm}, \eqref{eq:cseerror} that the distance to performative stable solution $\| \Tprm^t \|^2$ converges to zero as ${\cal O}(1/t)$, while the consensus error $\| \CSE{t} \|_F^2$ converges to zero as ${\cal O}(1/t^2)$.
We remark that our analysis also applies to the case of time varying graph; see \S\ref{sec:timevarying} for a proof sketch.
\vspace{.1cm}

On the other hand, the bound \eqref{eq:mainthm} can be simplified to
\beq \label{eq:net_ind}
\EE[ \| \Bprm^{t} - \thps \|^2 ] \lesssim  {\textstyle \prod_{i=1}^{t} \left( 1 - \frac{\tmu \gamma_{i}}{2} \right)} + \frac{L(\sigma^2 + \varsigma^2)}{ n \delta  \tmu \rho^2 \epsilon_{\sf avg} } \, \gamma_{t}^2 + \frac{\sigma^2}{n \tmu} \gamma_t.
\eeq 
As seen, the error is controlled by three terms. The first term is a transient term that consists of the product $\prod_{i=1}^t (1- \tmu \gamma_i)$. It decays sub-geometrically and is scaled by the initial error ${\sf D}$. The second term is a transient term and is affected by the spectral gap of mixing matrix $\rho$, the degree of heterogeneity $\varsigma$, etc. It decays as ${\cal O}(\gamma_t^2)$. Finally, the last term is a fluctuation term that only depends on the averaged noise variance ${\cal O}(\sigma^2 / n)$. It decays at the slowest rate as ${\cal O}(\gamma_t)$.

\textbf{Effects of Network Topology and Heterogeneity.} An interesting observation from \eqref{eq:net_ind} is on how the network topology ($\rho$) and heterogeneity across local sample distribution ($\varsigma$) affects the convergence behavior of {\bname}. First, we note that the last term of ${\cal O}( \gamma_t \, \sigma^2 / (\tmu n))$ is similar to the fluctuation term in SGD using the batch size of $n$ under a centralized setting, e.g., \citep{moulines2011non}. Second, the constants $\rho, \varsigma$ only affect the term of rate ${\cal O}(\gamma_t^2)$. 

In the case with diminishing step size  $\gamma_t = \frac{a_0}{a_1+t}$, the second term of \eqref{eq:net_ind} will vanish at a faster rate than the last term. Particularly, if for some ${\rm C} > 0$,
\beq \label{eq:transient}
\gamma_t \leq {\rm C} \cdot \delta \rho^2 \epsilon_{\sf avg} \sigma^2 [ L ( \sigma^2 + \varsigma^2 ) ]^{-1} 
\eeq
then \eqref{eq:net_ind} will be dominated by ${\cal O}( \gamma_t \, \sigma^2 /( \tmu n) )$. In other words, the effects of the network topology and heterogeneous population vanish asymptotically as $\Bprm^t \to \thps$.

As such, the `linear speedup' behavior of {DSGD} in \citep{lian2017decentralized, pu2021sharp} can be extended to {\aname} with the {\bname} scheme. Again, we highlight that the consensus seeking behavior of agents has led to such speedup in convergence towards the Multi-PS solution.\vspace{-.2cm}

% $\EE[ \| \Bprm^{t} - \thps \|^2 ]$ can be bounded by three terms. The first term 

\subsection{Proof of Theorem~\ref{thm1}}\vspace{-.1cm}
We outline the main steps in proving Theorem \ref{thm1}. 
As a preparatory step, we borrow the following lemma on the Lipschitzness of $\grd f_i (\prm, \prm)$ from \citep[Lemma 2.1]{drusvyatskiy2020stochastic}:
\begin{lemma}[{\bf Continuity of $\grd f_i$}]\label{lem: grd_dev}
Under A\ref{ass: smooth}, A\ref{ass:sensitive}. For any $\prm_0, \prm_1, \prm, \prm^\prime \in \RR^d$, it holds:
\begin{eqnarray}\label{lem: grd_ub_ineq}
    \begin{aligned}
        \norm{\grd f_{i}( \prm_0 ; \prm) - \grd f_{i}( \prm_1 ; \prm^\prime) }
        &\leq  L\norm{ \prm_0 - \prm_1 } + L \epsilon_i\norm{\prm-\prm^\prime}.
    \end{aligned}
\end{eqnarray}
\end{lemma}
By A\ref{ass: graph}, the update recursion for the average iterate of {\bname} can be expressed as
\beq \textstyle \label{eq:dsgd_avg}
\Bprm^{t+1} = (1/n)\sum_{i=1}^n \prm_i^{t+1} = \Bprm^t - \gamma_{t+1} \sum_{i=1}^n \grd \ell( \prm_i^t ; Z_i^{t+1} ) / n.
\eeq 
Using the above recursion,  we show the following lemma for the one-step progress of {\bname}:
\begin{lemma}\label{lem:descent}
{\bf (Descent Lemma)} Fix any $\delta > 0$ and let $\epsilon_{\sf avg}\leq \frac{\mu}{(1+\delta) L}$. Under A\ref{ass: strongcvx}, A\ref{ass: smooth}, A\ref{ass:sensitive}, A\ref{ass:SecOrdMom} and let the step sizes satisfy $\sup_{t \geq 0} \gamma_{t+1} \leq \frac{ \tmu }{ c_2 }$, the following bound holds
\begin{align}\label{lem:des_eq}
\EE_t \norm{\Tprm^{t+1}}^2 \leq ( 1 - \tmu \gamma_{t+1})\norm{\Tprm^t}^{2} + \left[c_1 \gamma_{t+1} + c_2 \gamma_{t+1}^2 \right] {\textstyle \frac{1}{n}} \norm{\CSE{t}}^2_F + \frac{2\sigma^2}{n}\gamma_{t+1}^2,
\end{align}
for any $t \geq 0$,
where $\EE_t[\cdot]$ is the expectation operator conditioned on the iterates up to the $t$th iteration, and we recall the definitions of $c_1, c_2, \tmu$ from \eqref{eq:c1c2c3}.
\end{lemma}
The proof is in \S\ref{appendix:lem_des}. 
We highlight that proving the upper bound \eqref{lem:des_eq} requires the smoothness property Lemma~\ref{lem: grd_dev} for handling the difference $\sum_{i=1}^{n}\grd f_i(\prm_i^t, \prm_i^t) - \grd f_i(\thps, \thps)$ as proportional to the error against the Multi-PS solution $\Tprm^t$ and the consensus error $\normtxt{ \CSE{t} }_F^2$. 
% Meanwhile, note that the free parameter $\delta$ in $\tmu$ adjusts the tightness of contraction after conducting one step gradient descent.
Lemma~\ref{lem:descent} prompts us to study the consensus error $\norm{\CSE{t}}_F^2$ and a key observation is:
\begin{lemma}\label{lem: consens}
{\bf (Consensus Error Bound)}
Under A\ref{ass: smooth}--A\ref{ass:hete} and let the step sizes satisfy $\sup_{t \geq 0} \gamma_{t+1} \leq {\rho}/\sqrt{2 c_3}$, then it holds 
\begin{align}\label{lem:consens_eq}
\EE_t[ {\textstyle \frac{1}{n}} \norm{\CSE{t+1}}^2_F] \leq \left(1 - \frac{\rho}{2} \right) {\textstyle \frac{1}{n}} \norm{\CSE{t}}_F^2 +  12 [ \sigma^2 + \varsigma^2 ] \frac{ \gamma_{t+1}^2 }{\rho } \norm{\Tprm^t}^2 
+ 9(\sigma^2+\varsigma^2)\frac{\gamma_{t+1}^2}{\rho},
\end{align}
for any $t \geq 0$,
where we recall that $c_3\eqdef 12 \sigma^2 + 18L^2(1+ \epsilon_{\sf max})^2$.
\end{lemma}
The proof is in \S\ref{appendix:lem_consens}. In \eqref{eq:cons_noA6} of the appendix, we provide an alternative consensus error bound without using A\ref{ass:hete}. Note that despite the decision dependent distributions due to the performative nature of {\aname}, the above bound shows a similar trend as in \citep{yes_topology, pu2021sharp, kong2021consensus,koloskova2019decentralized}.
% {\color{red} there should be more related works!}

However, unlike \citep[Lemma 2]{yes_topology}, the r.h.s.~of \eqref{lem:consens_eq} contains a ${\cal O}( [ \sigma^2 + \varsigma^2 ] \gamma_{t+1}^2\normtxt{\Tprm^t}^2 )$ term which arises from A\ref{ass:SecOrdMom}, A\ref{ass:hete} with the growth condition. This introduces new challenges to analysis as it will be insufficient to conclude from \eqref{lem:consens_eq} \emph{alone} that {\bname} converges to a consensual solution. 

Our plan is to consider Lemmas~\ref{lem:descent} and \ref{lem: consens} simultaneously in order to control $\EE \normtxt{\Tprm^t}^2$, $\EE \normtxt{ \CSE{t} }_F^2$. Define the following sequence of non-negative numbers: for any $t \geq 0$, 
\beq
\textstyle {\cal L}_{t+1} \eqdef \EE \big[ \normtxt{\Tprm^{t+1}}^2 +  \gamma_{t+1} \frac{8c_1}{\rho n} \norm{\CSE{t+1}}_F^2 \big].
\eeq 
We obtain the following lemma:
\begin{lemma}[{\bf Convergence of ${\cal L}_t$}]\label{lem:lya}
Under A\ref{ass: strongcvx}--A\ref{ass:hete}. Suppose that the step sizes satisfy $\sup_{t \geq 0} \gamma_{t+1} \leq \min \left\{ \frac{4}{\tmu},  \sqrt{ \frac{ \rho^2 \tmu }{ 192 c_1 ( \sigma^2 + \varsigma^2 ) } } , \frac{ \rho c_1 }{ 4 \tmu c_1 + \rho c_2 } \right\}$. For any $t \geq 0$, it holds 
\beq \label{eq:lem_lya_part0}
{\cal L}_{t+1} \leq (1 - \tmu \gamma_{t+1} / 2) \, {\cal L}_t + \rho^{-2}  72 c_1 (\sigma^2 + \varsigma^2)  \gamma_{t+1}^3 +  n^{-1} 2 \sigma^2 \gamma_{t+1}^2.
\eeq 
Further, if the step sizes satisfy $\frac{ \gamma_{t-1} }{ \gamma_{t} } \leq \min\{ \sqrt{1 + (\tmu/4) \gamma_t^2}, \sqrt[3]{ 1 + (\tmu/4) \gamma_t^3 } \}$ for any $t \geq 1$, then 
\beq \label{eq:lem_lya_fin}
\EE \left[ \norm{\Tprm^{t+1}}^2 +  \gamma_{t+1} \frac{8c_1}{\rho n} \norm{\CSE{t+1}}_F^2\right] \leq \prod_{i=1}^{t+1} \left( 1 - \frac{\tmu \gamma_i}{ 2 } \right) {\sf D} + \frac{ 288 c_1 ( \sigma^2 + \varsigma^2 ) }{ \rho^2 \tmu } \gamma_{t+1}^2 + \frac{ 8 \sigma^2 }{ \tmu n } \gamma_{t+1},
\eeq 
where we recall that ${\sf D} := \normtxt{\Tprm^0}^2 + \gamma_1 \frac{8c_1}{\rho n} \norm{\CSE{0}}_F^2$.
\end{lemma}
The proof is in \S\ref{app:lem_lya}, where \eqref{eq:lem_lya_part0} is based on a careful combination of Lemmas~\ref{lem:descent} and \ref{lem: consens}, and \eqref{eq:lem_lya_fin} is computed from the non-asymptotic analysis for the recursion \eqref{eq:lem_lya_part0}.\vspace{.1cm}

\textbf{Proof of Theorem~\ref{thm1}.}
Lemma~\ref{lem:lya} immediately leads to \eqref{eq:mainthm} of the theorem as the l.h.s.~of \eqref{eq:lem_lya_fin} is lower bounded by $\EE \normtxt{ \Tprm^{t+1} }^2 $.  
To obtain \eqref{eq:cseerror}, we observe from the simplifying the r.h.s.~of \eqref{eq:lem_lya_fin} that 
\beq 
\begin{split}
\sup_{t \geq 1} \EE \norm{\Tprm^t}^2 & \leq {\sf D} + 
\frac{ 288 c_1 ( \sigma^2 + \varsigma^2 ) }{ \rho^2 \tmu } \gamma_1^2 + \frac{ 8 \sigma^2 }{ \tmu n } \gamma_1 \leq {\sf D} + \frac{3}{2} + \frac{8 \sigma^2}{ c_2 n } =: \overline{\Delta}.
\end{split}
\eeq 
Note that $\normtxt{\Tprm^0}^2 \leq \overline{\Delta}$ as well.
Substituting the above into Lemma~\ref{lem: consens} yields
\beq 
\begin{aligned} 
{\textstyle \frac{1}{n}} \EE \norm{ \CSE{t+1} }_F^2 & \leq \left( 1 - {\rho} / {2} \right) {\textstyle \frac{1}{n}} \EE \norm{ \CSE{t} }_F^2 + \rho^{-1} (9 + 12 \overline{\Delta}) (\sigma^2+\varsigma^2) \gamma_{t+1}^2 \\
& \leq \left( 1 - \frac{\rho}{2} \right)^{t+1} {\textstyle \frac{1}{n}} \norm{ \CSE{0} }_F^2 + \frac{ (9 + 12 \overline{\Delta}) (\sigma^2+\varsigma^2) }{\rho} \sum_{s=1}^{t+1} \left( 1 - \frac{\rho}{2} \right)^{t+1-s} \gamma_{s}^2
\end{aligned}
\eeq 
Applying Lemma~\ref{lem:aux2} in the appendix together with the step size condition $\gamma_t / \gamma_{t+1} \leq 1 + \rho / (4-2\rho)$, $t \geq 1$, leads to \eqref{eq:cseerror} of the theorem. 
This concludes our proof. \hfill $\square$\vspace{-.2cm}

\section{Numerical Experiments}\label{sec:num}\vspace{-.2cm}
We consider two examples of performative prediction problems to verify our theories. All experiments are conducted with Python on a server using 80 threads of an Intel Xeon 6318 CPU.\vspace{.1cm}

\textbf{Multi-agent Gaussian Mean Estimation.} We aim to illustrate Proposition~\ref{lem:exist}, Theorem~\ref{thm1} via a scalar Gaussian mean estimation problem on synthetic data. We consider $n=25$ agents connected on a ring graph, and the {\aname} problem \eqref{eq:multipfd} is specified with $\ell(\prm_i; Z_i) = (\prm_i - Z_{i})^2/2$. The local distributions are given by ${\cal D}_{i}(\prm_i)\equiv {\cal N}(\bar{z}_i + \epsilon_i \prm_i, \sigma^2)$, where $\bar{z}_{i}$ is the mean value to be estimated. 
For this problem, we have $\mu = 1$, $L=1$, as such if $0< \bar{\epsilon} = \epsilon_{\sf avg} < 1$, the Multi-PS solution can be computed in closed form as $\thps  = \sum_{i=1}^{n}\bar{z}_i / [n(1-\epsilon_{\sf avg})]$; while $\thps$ does not exist if $\epsilon_{\sf avg} \geq 1$. 

In our experiments, we set $\bar{z}_{i} = 10$, $\sigma^2 = 50$ and step size for {\bname} as $\gamma_{t} = {a_0}/{(a_1 + t)}$ with $a_0 = 50$, $a_1 = 10^4$. 
In Fig.~\ref{fig:me_plot}, we compare the gap $\normtxt{\Bprm^t - \thps}^2$, consensus error $\normtxt{\CSE{t}}^2$, expected performative risk $f(\Bprm^t; \Bprm^t)$ of \eqref{eq:multipfd}, against the iteration number $t$. We examine the behavior of {\bname} when the {\aname} problem has an averaged sensitivity parameter of $\epsilon_{\sf avg} \in \{ 0.9, 1.01, 1.05, 1.1 \}$ and under a heterogeneous, decision-dependent distribution environment where $\epsilon_i$ are distinct.

% Set $\epsilon_{\sf avg} = 0.9 < \mu /L = 1$ and  sensitivity parameter of each local data shift is 
% $\bm{\epsilon} = [0.4, 0.45, \ldots, 0.95, 1.025, 1.05, \ldots, 1.3]$.
% We set $\bar{z}_{i} = 10, \sigma^2 = 50$. The step size is 
We first observe from Fig. \ref{fig:me_plot} (left) and (middle) that when $\epsilon_{\sf avg} = 0.9 < 1$, the gap $\normtxt{\Bprm^t - \thps}^2$ decays at ${\cal O}(1/t)$ as $t \to \infty$, while the consensus error $\normtxt{\CSE{t}}^2$ decays at ${\cal O}(1/t^2)$. This coincides exactly with Theorem~\ref{thm1}. 
Furthermore, in Fig.~\ref{fig:me_plot} (right), we simulate a setting when one of the agents with $\epsilon_i = 1.01$ is always disconnected from the network and perform the greedy deployment scheme \emph{individually}. We observe from the figure that its performative risk $f_i( \prm_i^t ; \prm_i^t )$ diverges as $t \to \infty$. This indicates that consensus can help stabilize the system. 

Lastly, from Fig.~\ref{fig:me_plot} (middle) and (right), we observe that whenever $\epsilon_{\sf avg} > 1$, the consensus error and performative risk diverge. Again, this corroborates with our Proposition~\ref{lem:exist}.\vspace{.1cm}
% but as long as $\epsilon_{avg}<1$, the averaged performative stable gap $\normtxt{\Bprm^t-\thps}^2$ still can converge, which verify Proposition \ref{lem:exist}, i.e., the convergence region $\epsilon_{avg}<\mu/L$ is tight.
% {\color{red} elaborate more: e.g., $\thps$ is converging with $1/t$, while consensus error is $1/t^2$; whenever $\epsilon_{\sf avg} > 1$, it diverges, or when $\epsilon_i > 1$, the individual diverges}

\begin{figure}
\centering
\includegraphics[width=.985\textwidth]{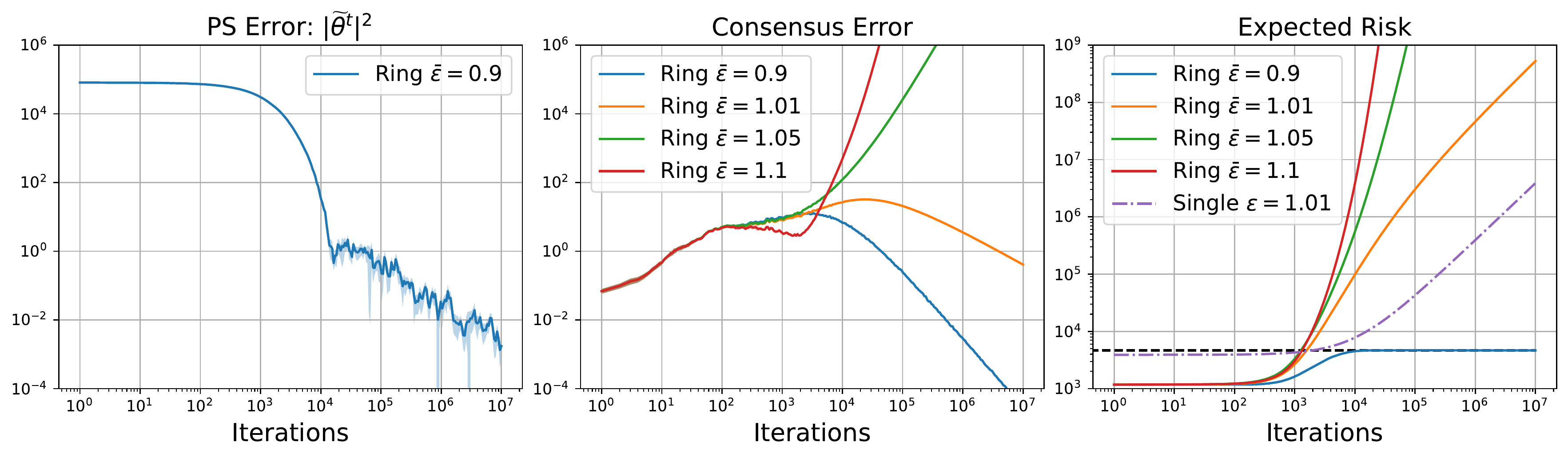}\vspace{-.2cm}
\caption{\textbf{Multi-agent Gaussian Mean Estimation.} 
% Convergence to $\thps$ for ${\epsilon}_{\sf avg} < 1$ and divergence for $\bar{\epsilon}>1$. Horizontal dashed line shows the expected risk of $\thps$ on the 3rd plot. Single stands for the case of a local Gaussian mean estimation. 
% % For Ring $\bar{\epsilon}=0.9$, the distribution of $\epsilon_i$ is \{0.4, 0.45, 0.45, 0.5, 0.55, 0.6, 0.65, 0.7, 0.75, 0.8, 0.85, 0.9, 0.95, 1.025, 1.05, 1.075, 1.1, 1.125, 1.15, 1.175, 1.2, 1.225, 1.25, 1.275, 1.3\}.
(Left) Gap to Multi-PS solution. Note that $\thps$ does not exist if $\epsilon_{\sf avg} \geq 1$ and the plots are thus skipped. (Middle) Consensus error $\normtxt{ \CSE{t} }_F^2$. (Right) Performative risk \eqref{eq:perfrisk}.  Results are averaged over 10 runs and shaded area is 90\% confidence interval.}\vspace{-.2cm}
\label{fig:me_plot}
\end{figure}

%%%%%%%%%%%%%%
\textbf{Email Spam Classification.} {We evaluate the performance of {\tt DSGD-GD} by simulating the performative effects on a real dataset}. This example is a multi-agent spam classification task based on {\tt spambase}, a dataset \citep{spambase94} with $m=4601$ samples, $d=48$ features. We adopt Example~\ref{example1} and simulate a scenario with 25 regional servers on a ring graph. Each server has access to training data from $m_i=138$ samples from {\tt spambase} modeling the different set of users; the rest of $m_{\sf train} = 1150$ samples are taken as testing data. The servers aim to find a common \emph{spam filter classifier} via \eqref{eq:log_loss} with $\beta=10^{-4}$. To model the strategic behavior of users, their features ${\bm X}_i$ are adapted to $\prm_i$ through maximizing a linear utility function, resulting in the shifted distribution ${\cal D}_i(\prm_i)$ specified in \eqref{eq:perf_quad}. The sensitivity parameters are set as $\epsilon_i \in \{0.4 \epsilon_{\sf avg}, 0.45 \epsilon_{\sf avg}, \dots, 1.6 \epsilon_{\sf avg}\} $ with $\bar{\epsilon} = \epsilon_{\sf avg} \in \{ 0.01, 0.1, 1\}$.\vspace{-.1cm}
% will try to react to the classifier linearly in order to split past the spam filter, where $z\equiv (x,y)$, where $x\in \RR^d$ is feature vector, label $y = \mathbf{1}\{\text{email is spam}\}$, 0 otherwise. Problem (\ref{eq:multipfd})  is specified as in Example~\ref{example1}. The shifted data distribution ${\cal D}_{i}(\prm_i)$ is obtained through evaluating a quadratic response (\ref{eq:perf_quad}). 
% Such a distribution map describes the scenario where users (including spammers) optimize a non-spam classification utility minus a quadratic cost for changing their features. 

Our results are shown in Fig.~\ref{fig:email_spam} as we compare the gradient $\|\grd f( \Bprm^t ; \Bprm^t )\|^2$ evaluated on the training dataset, the consensus error $\normtxt{\CSE{t}}_F^2$, and the accuracy on the testing dataset, against the iteration number $t$. From Fig.~\ref{fig:email_spam} (left) and (middle), we observe that the {\bname} scheme converges to the Multi-PS solution and reaches consensus under various settings of $\epsilon_{\sf avg}$, at the rates ${\cal O}(1/t), {\cal O}(1/t^2)$, respectively. 
% This corroborates with our Theorem~\ref{thm1}.
In Fig.~\ref{fig:email_spam} (right), we evaluate the performance of the trained classifier $\prm_i^t$ on the testing dataset with shifted distribution due to $\prm_i^t$. We compare with a non-performatively trained solution obtained by solving $\prm^\star = \argmin_{ \prm } \sum_{i=1}^n f_i ( \prm ; {\bm 0} )$ ({\tt DSGD} in the legend), i.e., without any shift in distributions, but evaluate the performance on distribution shifted by $\prm^\star$.  As observed, the test accuracy decreases as sensitivity $\epsilon_{\sf avg}$ increases, and {\bname} achieves better accuracy than {\tt DSGD}. 

\begin{figure}
    \centering
    \includegraphics[width=.985\textwidth]{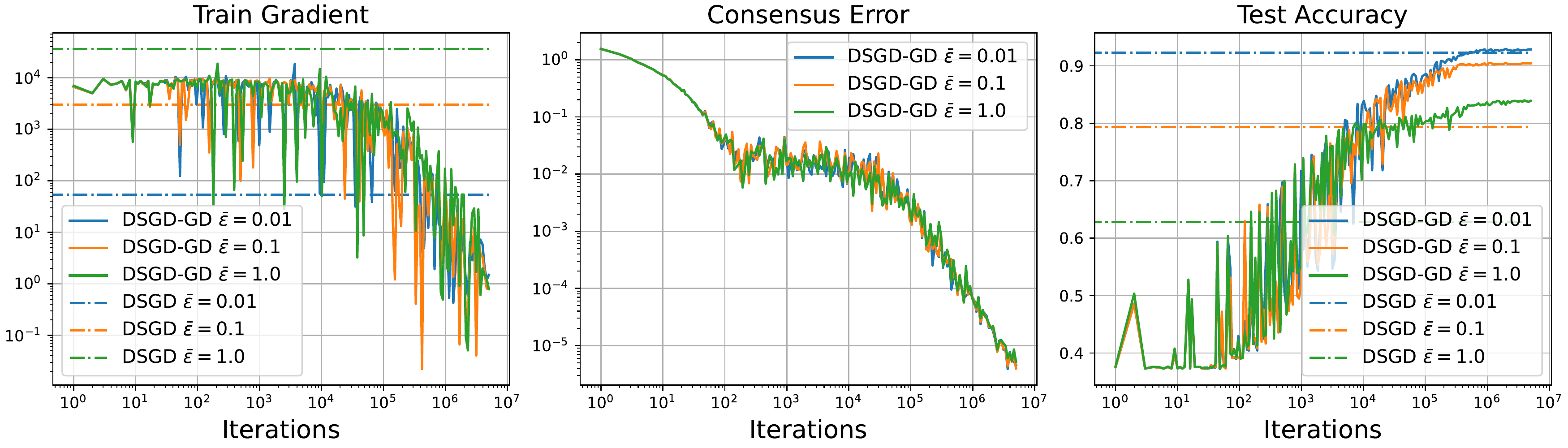}\vspace{-.2cm}
    \caption{\textbf{Spam Email Classification.} 
    (Left) Gradient on training dataset $\|\grd f( \Bprm^t ; \Bprm^t )\|^2$. Note that $\grd f( \thps; \thps ) = {\bm 0}$ and thus the gradient norm measures the gap to $\thps$. (Middle) Consensus Error. (Right) Test accuracy with shifted distributions. We also compare the non-performative optimal solution (dashed lines, {\tt DSGD} in the legend) on the shifted dataset.}\vspace{-.3cm}
    \label{fig:email_spam}
\end{figure}

%%%%%%%%%%%%%%

% \vspace{.1cm}

\textbf{Conclusions \& Limitations.} In this paper, we studied the {\aname} problem, and analyzed its stability when a {\bname} scheme is applied. Our results indicate that when agents are \emph{consensus seeking}, {\aname} admits a performative stable solution with laxer condition and {\bname} achieves linear speedup. Limitations to our current results include the requirement of synchronous updates among agents, strongly convex loss [cf.~A\ref{ass: strongcvx}], etc., which shall be explored in future extensions.

%%%%%%%%%%%%%%%%%%%%%%%%%%%%%%%%%%%%%%%%%%%%%%%%%%%%%%%%%%%%%%%%%%%%%%%%%%%%%%%%%%%%%%%%%%%%%%%%%%%%%%%%%%%%%%%

\newpage
\bibliographystyle{plainnat}
\bibliography{ecl}

\begin{thebibliography}{38}
\providecommand{\natexlab}[1]{#1}
\providecommand{\url}[1]{\texttt{#1}}
\expandafter\ifx\csname urlstyle\endcsname\relax
  \providecommand{\doi}[1]{doi: #1}\else
  \providecommand{\doi}{doi: \begingroup \urlstyle{rm}\Url}\fi

\bibitem[Bars et~al.(2022)Bars, Bellet, Tommasi, and Kermarrec]{yes_topology}
B.~Le Bars, A.~Bellet, M.~Tommasi, and AM. Kermarrec.
\newblock Yes, topology matters in decentralized optimization: Refined
  convergence and topology learning under heterogeneous data.
\newblock \emph{arXiv preprint arxiv:2204.04452}, 2022.
\newblock URL \url{https://arxiv.org/abs/2204.04452}.

\bibitem[Bianchi and Jakubowicz(2012)]{bianchi2012convergence}
Pascal Bianchi and J{\'e}r{\'e}mie Jakubowicz.
\newblock Convergence of a multi-agent projected stochastic gradient algorithm
  for non-convex optimization.
\newblock \emph{IEEE transactions on automatic control}, 58\penalty0
  (2):\penalty0 391--405, 2012.

\bibitem[Boyd et~al.(2004)Boyd, Diaconis, and Xiao]{boyd2004fastest}
Stephen Boyd, Persi Diaconis, and Lin Xiao.
\newblock Fastest mixing markov chain on a graph.
\newblock \emph{SIAM review}, 46\penalty0 (4):\penalty0 667--689, 2004.

\bibitem[Brown et~al.(2022)Brown, Hod, and Kalemaj]{brown2020performative}
Gavin Brown, Shlomi Hod, and Iden Kalemaj.
\newblock Performative prediction in a stateful world.
\newblock In \emph{AISTATS}, 2022.

\bibitem[Caldas et~al.(2019)Caldas, Duddu, Wu, Li, Konečný, McMahan, Smith,
  and Talwalkar]{caldas2019leaf}
Sebastian Caldas, Sai Meher~Karthik Duddu, Peter Wu, Tian Li, Jakub Konečný,
  H.~Brendan McMahan, Virginia Smith, and Ameet Talwalkar.
\newblock Leaf: A benchmark for federated settings.
\newblock In \emph{NeurIPS Workshop on Federated Learning for Data Privacy and
  Confidentiality}, 2019.

\bibitem[Chen et~al.(2021)Chen, Zhang, Giannakis, and
  Basar]{chen2021communication}
Tianyi Chen, Kaiqing Zhang, Georgios~B Giannakis, and Tamer Basar.
\newblock Communication-efficient policy gradient methods for distributed
  reinforcement learning.
\newblock \emph{IEEE Transactions on Control of Network Systems}, 2021.

\bibitem[Dong et~al.(2018)Dong, Roth, Schutzman, Waggoner, and
  Wu]{dong2018strategic}
Jinshuo Dong, Aaron Roth, Zachary Schutzman, Bo~Waggoner, and Zhiwei~Steven Wu.
\newblock Strategic classification from revealed preferences.
\newblock In \emph{Proceedings of the 2018 ACM Conference on Economics and
  Computation}, pages 55--70, 2018.

\bibitem[Drusvyatskiy and Xiao(2020)]{drusvyatskiy2020stochastic}
Dmitriy Drusvyatskiy and Lin Xiao.
\newblock Stochastic optimization with decision-dependent distributions.
\newblock \emph{ArXiv preprint arxiv:2011.11173}, 2020.

\bibitem[Granas and Dugundji(2003)]{granas2003fixed}
Andrzej Granas and James Dugundji.
\newblock \emph{Fixed point theory}, volume~14.
\newblock Springer, 2003.

\bibitem[Hardt et~al.(2016)Hardt, Megiddo, Papadimitriou, and Wootters]{Hardt}
Moritz Hardt, Nimrod Megiddo, Christos Papadimitriou, and Mary Wootters.
\newblock Strategic classification.
\newblock In \emph{ITCS}, page 111–122, 2016.

\bibitem[Hopkins(1999)]{spambase94}
Reeber Hopkins, Mark.
\newblock {Spambase}.
\newblock UCI Machine Learning Repository, 1999.

\bibitem[Izzo et~al.(2021)Izzo, Ying, and Zou]{izzo2021learn}
Zachary Izzo, Lexing Ying, and James Zou.
\newblock How to learn when data reacts to your model: performative gradient
  descent.
\newblock In \emph{ICML}, pages 4641--4650. PMLR, 2021.

\bibitem[Izzo et~al.(2022)Izzo, Zou, and Ying]{izzo2022learn}
Zachary Izzo, James Zou, and Lexing Ying.
\newblock How to learn when data gradually reacts to your model.
\newblock In \emph{AISTATS}, 2022.

\bibitem[Karimi et~al.(2019)Karimi, Miasojedow, Moulines, and
  Wai]{karimi2019non}
Belhal Karimi, Blazej Miasojedow, Eric Moulines, and Hoi-To Wai.
\newblock Non-asymptotic analysis of biased stochastic approximation scheme.
\newblock In \emph{Conference on Learning Theory}, pages 1944--1974. PMLR,
  2019.

\bibitem[Kleinberg and Raghavan(2020)]{kleinberg2020classifiers}
Jon Kleinberg and Manish Raghavan.
\newblock How do classifiers induce agents to invest effort strategically?
\newblock \emph{ACM Transactions on Economics and Computation (TEAC)},
  8\penalty0 (4):\penalty0 1--23, 2020.

\bibitem[Koloskova et~al.(2019)Koloskova, Stich, and
  Jaggi]{koloskova2019decentralized}
Anastasia Koloskova, Sebastian Stich, and Martin Jaggi.
\newblock Decentralized stochastic optimization and gossip algorithms with
  compressed communication.
\newblock In \emph{International Conference on Machine Learning}, pages
  3478--3487. PMLR, 2019.

\bibitem[Kong et~al.(2021)Kong, Lin, Koloskova, Jaggi, and
  Stich]{kong2021consensus}
Lingjing Kong, Tao Lin, Anastasia Koloskova, Martin Jaggi, and Sebastian Stich.
\newblock Consensus control for decentralized deep learning.
\newblock In \emph{International Conference on Machine Learning}, pages
  5686--5696. PMLR, 2021.

\bibitem[Lan et~al.(2020)Lan, Lee, and Zhou]{lan2020communication}
Guanghui Lan, Soomin Lee, and Yi~Zhou.
\newblock Communication-efficient algorithms for decentralized and stochastic
  optimization.
\newblock \emph{Mathematical Programming}, 180\penalty0 (1):\penalty0 237--284,
  2020.

\bibitem[Li and Wai(2022)]{li2021state}
Qiang Li and Hoi-To Wai.
\newblock State dependent performative prediction with stochastic
  approximation.
\newblock In \emph{AISTATS}, 2022.

\bibitem[Lian et~al.(2017)Lian, Zhang, Zhang, Hsieh, Zhang, and
  Liu]{lian2017decentralized}
Xiangru Lian, Ce~Zhang, Huan Zhang, Cho-Jui Hsieh, Wei Zhang, and Ji~Liu.
\newblock Can decentralized algorithms outperform centralized algorithms? a
  case study for decentralized parallel stochastic gradient descent.
\newblock In \emph{NeurIPS}, 2017.

\bibitem[Mendler-D{\"u}nner et~al.(2020)Mendler-D{\"u}nner, Perdomo, Zrnic, and
  Hardt]{mendler2020}
Celestine Mendler-D{\"u}nner, Juan Perdomo, Tijana Zrnic, and Moritz Hardt.
\newblock Stochastic optimization for performative prediction.
\newblock \emph{Advances in Neural Information Processing Systems},
  33:\penalty0 4929--4939, 2020.

\bibitem[Miller et~al.(2021)Miller, Perdomo, and Zrnic]{miller2021}
John Miller, Juan~C. Perdomo, and Tijana Zrnic.
\newblock Outside the echo chamber: Optimizing the performative risk.
\newblock In \emph{ICML}, 2021.

\bibitem[Moulines and Bach(2011)]{moulines2011non}
Eric Moulines and Francis Bach.
\newblock Non-asymptotic analysis of stochastic approximation algorithms for
  machine learning.
\newblock In \emph{Advances in neural information processing systems},
  volume~24, 2011.

\bibitem[Narang et~al.(2022)Narang, Faulkner, Drusvyatskiy, Fazel, and
  Ratliff]{narang2022multiplayer}
Adhyyan Narang, Evan Faulkner, Dmitriy Drusvyatskiy, Maryam Fazel, and
  Lillian~J. Ratliff.
\newblock Multiplayer performative prediction: Learning in decision-dependent
  games.
\newblock In \emph{AISTATS}, 2022.

\bibitem[Perdomo et~al.(2020)Perdomo, Zrnic, Mendler-D{\"u}nner, and
  Hardt]{perdomo2020performative}
Juan Perdomo, Tijana Zrnic, Celestine Mendler-D{\"u}nner, and Moritz Hardt.
\newblock Performative prediction.
\newblock In \emph{International Conference on Machine Learning}, pages
  7599--7609. PMLR, 2020.

\bibitem[Piliouras and Yu(2022)]{piliouras2022multi}
Georgios Piliouras and Fang-Yi Yu.
\newblock Multi-agent performative prediction: From global stability and
  optimality to chaos.
\newblock \emph{arXiv preprint arXiv:2201.10483}, 2022.

\bibitem[Pu et~al.(2021)Pu, Olshevsky, and Paschalidis]{pu2021sharp}
Shi Pu, Alexander Olshevsky, and Ioannis~Ch Paschalidis.
\newblock A sharp estimate on the transient time of distributed stochastic
  gradient descent.
\newblock \emph{IEEE Transactions on Automatic Control}, 2021.

\bibitem[Qui{\~n}onero-Candela et~al.(2008)Qui{\~n}onero-Candela, Sugiyama,
  Schwaighofer, and Lawrence]{quinonero2008dataset}
Joaquin Qui{\~n}onero-Candela, Masashi Sugiyama, Anton Schwaighofer, and Neil~D
  Lawrence.
\newblock \emph{Dataset shift in machine learning}.
\newblock MIT Press, 2008.

\bibitem[Ram et~al.(2010)Ram, Nedi{\'c}, and
  Veeravalli]{sundhar2010distributed}
S~Sundhar Ram, Angelia Nedi{\'c}, and Venugopal~V Veeravalli.
\newblock Distributed stochastic subgradient projection algorithms for convex
  optimization.
\newblock \emph{Journal of optimization theory and applications}, 147\penalty0
  (3):\penalty0 516--545, 2010.

\bibitem[Ray et~al.(2022)Ray, Ratliff, Drusvyatskiy, and
  Fazel]{ray2022decision}
Mitas Ray, Lillian~J Ratliff, Dmitriy Drusvyatskiy, and Maryam Fazel.
\newblock Decision-dependent risk minimization in geometrically decaying
  dynamic environments.
\newblock In \emph{AAAI Conference on Artificial Intelligence, To appear},
  2022.

\bibitem[Sayed(2014)]{sayed2014adaptation}
Ali~H Sayed.
\newblock Adaptation, learning, and optimization over networks.
\newblock \emph{Foundations and Trends in Machine Learning}, 7:\penalty0
  311--801, 2014.

\bibitem[Tang et~al.(2018)Tang, Lian, Yan, Zhang, and Liu]{tang2018d}
Hanlin Tang, Xiangru Lian, Ming Yan, Ce~Zhang, and Ji~Liu.
\newblock $d^2$: Decentralized training over decentralized data.
\newblock In \emph{International Conference on Machine Learning}, pages
  4848--4856. PMLR, 2018.

\bibitem[Wood et~al.(2021)Wood, Bianchin, and Dall’Anese]{wood2021online}
Killian Wood, Gianluca Bianchin, and Emiliano Dall’Anese.
\newblock Online projected gradient descent for stochastic optimization with
  decision-dependent distributions.
\newblock \emph{IEEE Control Systems Letters}, 6:\penalty0 1646--1651, 2021.

\bibitem[Yuan and Alghunaim(2021)]{yuan2021removing}
Kun Yuan and Sulaiman~A Alghunaim.
\newblock Removing data heterogeneity influence enhances network topology
  dependence of decentralized sgd.
\newblock \emph{arXiv preprint arXiv:2105.08023}, 2021.

\bibitem[Zhang et~al.(2018)Zhang, Yang, Liu, Zhang, and Basar]{zhang2018fully}
Kaiqing Zhang, Zhuoran Yang, Han Liu, Tong Zhang, and Tamer Basar.
\newblock Fully decentralized multi-agent reinforcement learning with networked
  agents.
\newblock In \emph{International Conference on Machine Learning}, pages
  5872--5881. PMLR, 2018.

\bibitem[Zhang et~al.(2020)Zhang, Koppel, Zhu, and Basar]{zhang2020global}
Kaiqing Zhang, Alec Koppel, Hao Zhu, and Tamer Basar.
\newblock Global convergence of policy gradient methods to (almost) locally
  optimal policies.
\newblock \emph{SIAM Journal on Control and Optimization}, 58\penalty0
  (6):\penalty0 3586--3612, 2020.

\bibitem[Zhang et~al.(2021)Zhang, Yang, and Ba{\c{s}}ar]{zhang2021multi}
Kaiqing Zhang, Zhuoran Yang, and Tamer Ba{\c{s}}ar.
\newblock Multi-agent reinforcement learning: A selective overview of theories
  and algorithms.
\newblock \emph{Handbook of Reinforcement Learning and Control}, pages
  321--384, 2021.

\bibitem[Zrnic et~al.(2021)Zrnic, Mazumdar, Sastry, and Jordan]{zrnic2021leads}
Tijana Zrnic, Eric Mazumdar, Shankar Sastry, and Michael Jordan.
\newblock Who leads and who follows in strategic classification?
\newblock In \emph{NeurIPS}, volume~34, 2021.

\end{thebibliography}

\newpage

\newpage

\appendix 

\section{Proof of Proposition~\ref{lem:exist}} \label{app:exist}

% \begin{proof}[Proof of Lemma \ref{lem:exist}] 
Fix any $\prm^\prime, \prm \in \RR^d$. The optimality condition to \eqref{eq:map_M} implies that 
\beq \textstyle 
    \sum_{i=1}^{n}\grd  f_{i}({\cal M}(\prm); \prm) = {\bm 0}, \qquad
    \sum_{i=1}^{n}\grd f_{i}({\cal M}(\prm^{\prime}); \prm^{\prime}) = {\bm 0}.
\eeq 
Note that the gradients are taken w.r.t.~the first argument in the function $f_i$.
Observe the chain 
\begin{align*}
        0 & = \pscal{ {\bm 0} }{ {\cal M}( \prm ) - {\cal M} (\prm^\prime) } =  \Pscal{\sum_{i=1}^{n}\left[ \grd f_{i}(\mathcal{M}(\prm); \prm) - \grd f_{i}({\cal M}(\prm^\prime); \prm^{\prime})\right] }{\mathcal{M}(\prm)-{\cal M}(\prm^{\prime})}.    
\end{align*}
Adding and subtracting $\sum_{i=1}^n \grd f_{i}({\cal M}(\prm); \prm^{\prime})$ implies the equality:
\beq \label{eq:equality_stable}
\begin{aligned}
    & \textstyle \sum_{i=1}^{n}\Pscal{ \grd f_i(\mathcal{M}(\prm); \prm^\prime) - \grd f_i(\mathcal{M}(\prm) ; \prm) }{ \mathcal{M}(\prm) - {\cal M}(\prm^{\prime})}
    \\
    &= \textstyle 
    \sum_{i=1}^{n}\Pscal{ \left( \grd f_i(\mathcal{M}(\prm) ; \prm^\prime)-\grd f_i(\mathcal{M}(\prm^\prime) ; \prm^\prime) \right) }{ \mathcal{M}(\prm)-{\cal M}(\prm^{\prime})}.
\end{aligned}
\eeq 
Applying A\ref{ass: strongcvx} to the right hand side of \eqref{eq:equality_stable} lead to:
\beqq %\textstyle 
\sum_{i=1}^{n}\Pscal{ \left( \grd f_i(\mathcal{M}(\prm) ; \prm^\prime)-\grd f_i(\mathcal{M}(\prm^\prime) ; \prm^\prime) \right) }{ \mathcal{M}(\prm)-{\cal M}(\prm^{\prime})} \geq n \mu \| {\cal M}( \prm ) - {\cal M} ( \prm^\prime ) \|^2.
\eeqq 
Meanwhile, applying 
% A\ref{ass: smooth}, A\ref{ass:sensitive}
Lemma \ref{lem: grd_dev} to the left hand side of \eqref{eq:equality_stable} gives
\beqq
\begin{aligned}
& \textstyle \sum_{i=1}^{n}\Pscal{ \grd f_i(\mathcal{M}(\prm); \prm^\prime) - \grd f_i(\mathcal{M}(\prm) ; \prm) }{ \mathcal{M}(\prm) - {\cal M}(\prm^{\prime})} \\
& \textstyle \leq \sum_{i=1}^n \epsilon_i L \| \prm^\prime - \prm \| \| {\cal M}( \prm ) - {\cal M} ( \prm^\prime ) \|.
\end{aligned}
\eeqq
Substituting back into \eqref{eq:equality_stable} implies that 
\beq\label{eq:contraction}
\| {\cal M}( \prm ) - {\cal M} ( \prm^\prime ) \| \leq \frac{ \sum_{i=1}^n \epsilon_i L }{ n \mu } \| \prm - \prm^\prime \| = \frac{ \epsilon_{\sf avg} L }{ \mu } \| \prm - \prm^\prime \|.
\eeq
Therefore, the map ${\cal M}: \RR^d \to \RR^d$ is a contraction if $\epsilon_{\sf avg} < \mu / L$.
Subsequently, by the Banach fixed point theorem \citep{granas2003fixed}, the map ${\cal M}(\prm)$ admits a unique fixed point which is denoted as $\thps$.

% \aaa{If not, we assume that there exist another fixed point denoted as $\hat{\prm}^{PS}$ such that ${\cal M}(\hat{\prm}^{PS}) = \hat{\prm}^{PS}$ and $\thps \neq \hat{\prm}^{PS}$. From (\ref{eq:contraction}), we have
% \[
% \| {\cal M}( \thps ) - {\cal M} ( \hat{\prm}^{PS} ) \|\leq  \frac{ \epsilon_{\sf avg} L }{ \mu } \| \thps - \hat{\prm}^{PS} \|.
% \]
% which is equivalent to
% \[
% \|  \thps  - \hat{\prm}^{PS}  \|\leq  \frac{ \epsilon_{\sf avg} L }{ \mu } \| \thps - \hat{\prm}^{PS} \|.
% \]
% That is to say,
% \[
%     \frac{ \epsilon_{\sf avg} L }{ \mu } \geq 1
% \]
% which derives a contradiction. Therefore, we conclude that the map ${\cal M}(\cdot)$ admits a unique fixed point.}

To prove the converse, we consider the following instantiation of \eqref{eq:map_M} with
\beq 
\ell( \theta; Z ) = \frac{1}{2} ( \theta - Z )^2, \quad Z \sim {\cal D}_i( \theta ) \Longleftrightarrow Z \sim {\cal N}( \mu_i + \epsilon_i \theta , 1 )
\eeq 
Note that the above satisfies A\ref{ass: strongcvx} with $\mu = 1$, A\ref{ass: smooth} with $L=1$, A\ref{ass:sensitive} with $\epsilon_i$ for $i=1,\ldots,n$. 
We consider a case where it holds $\epsilon_{\sf avg} \geq \mu/L = 1$. We also let $\mu_{\sf avg} := (1/n) \sum_{i=1}^n \mu_i \neq 0$. 

We observe
\beq 
\begin{aligned}
f_i ( \theta' ; \theta ) & = \EE_{ Z \sim {\cal D}_i( \theta) } \left[ \frac{1}{2} ( \theta' - Z)^2 \right] = \EE_{ \tilde{Z} \sim {\cal N}(0,1) } \left[ \frac{1}{2} ( \theta' - \mu_i - \epsilon_i \theta - \tilde{Z} )^2 \right] \\
& = \frac{1}{2}  ( \theta' - \mu_i - \epsilon_i \theta )^2 + \frac{1}{2}.
\end{aligned}
\eeq 
For any $\theta \in \RR$, it can be shown that
\beq 
{\cal M}( \theta ) = \argmin_{ \theta' \in \RR } ~\frac{1}{2n} \sum_{i=1}^n ( \theta' - \mu_i - \epsilon_i \theta )^2 = \epsilon_{\sf avg} \theta + 
\mu_{\sf avg}
\eeq
Thus, applying the map for $T$ times leads to 
\beq 
{\cal M}^T( \theta ) = \epsilon_{\sf avg}^T \theta + \left( 1 + \epsilon_{\sf avg} + \cdots + \epsilon_{\sf avg}^{T-1} \right) \mu_{\sf avg}
\eeq 
Since $\epsilon_{\sf avg} > 1$ and $\mu_{\sf avg} \neq 0$, we have $\lim_{T \to \infty} |{{\cal M}^T( \theta )}| = \infty$ and the map is not a contraction.

%%%%%%%%%%%%%%%%%%%%%%%
\section{\bf Proof of Lemma \ref{lem:descent}}\label{appendix:lem_des}
Recall that $\Tprm^t := \Bprm^t - \prm^{PS}$ is the error of averaged decision at the $t$th iteration. Using \eqref{eq:dsgd_avg}, we have
\begin{align}\label{rec1}
    \norm{ \Tprm^{t+1} }^2 &= \norm{ \Tprm^{t} }^2 - \frac{2\gamma_{t+1}}{n} \Pscal{\Tprm^t}{ \sum_{i=1}^{n}\grd \ell(\prm_i^t; Z_i^{t+1})} + \frac{\gamma_{t+1}^2}{n^2} \norm{\sum_{i=1}^{n}\grd \ell(\prm_i^t; Z_i^{t+1})}^2.
\end{align}
We consider taking the conditional expectation $\EE_{t}[\cdot]$ on the both sides. Using the fixed point condition $\sum_{i=1}^{n}\grd f_i(\thps; \thps) = {\bm 0}$, we observe the following equivalent expression for the last term 
\beq \notag
\norm{\sum_{i=1}^{n} \grd \ell(\prm_i^t; Z_i^{t+1})}^2 = \norm{ \sum_{i=1}^{n} \big[ \grd \ell(\prm_i^t; Z_i^{t+1}) - \grd f_i( \prm_i^t; \prm_i^t ) + \grd f_i( \prm_i^t; \prm_i^t ) - \grd f_i( \thps; \thps)  \big] }^2
\eeq 
Observe that $Z_i^{t+1}$, $i=1,\ldots,n$ are independent r.v.s, taking the conditional expectation $\EE_t[\cdot]$ yields the upper bound to the above term
\beq \label{eq:lem3key1}
\begin{aligned}
& \EE_t \norm{\sum_{i=1}^{n} \grd \ell(\prm_i^t; Z_i^{t+1})}^2 \\
& \leq  2 \sum_{i=1}^n \EE_t \norm{ \grd \ell(\prm_i^t; Z_i^{t+1}) - \grd f_i( \prm_i^t; \prm_i^t ) }^2 + 2n \sum_{i=1}^n \norm{ \grd f_i( \prm_i^t; \prm_i^t ) - \grd f_i( \thps; \thps) }^2 \\
& \leq 2 \sum_{i=1}^n \sigma^2 ( 1 + \norm{ \prm_i^t - \thps }^2 ) + 2n \sum_{i=1}^n L^2 (1 + \epsilon_i)^2 \norm{ \prm_i^t - \thps }^2 \\
& \leq 2 \sigma^2 n + 4n [ \sigma^2 + n L^2 ( 1 + \epsilon_{\sf max} )^2 ] \norm{ \Tprm^t }^2 + 4 [ \sigma^2 + n L^2 ( 1 + \epsilon_{\sf max} )^2 ] \norm{ \CSE{t} }_F^2  
\end{aligned}
\eeq 
where the first inequality is due to A\ref{ass:SecOrdMom} and Lemma~\ref{lem: grd_dev}.
We conclude that
\beq \label{eq:ub} 
\begin{aligned}
& \frac{1}{n^2} \EE_t \norm{\sum_{i=1}^{n}\grd \ell(\prm_i^t; Z_i^{t+1})}^2 \leq 
\frac{ 2 \sigma^2 }{n} + c_2 \norm{ \Tprm^t }^2 + c_2 \frac{1}{n} \norm{ \CSE{t} }_F^2  
\end{aligned}
\eeq 
where we recall the definition that $c_2 = 4 \left( \frac{\sigma^2}{n} + L^2( 1 + \epsilon_{\sf max} )^2 \right)$.

% \begin{eqnarray}\label{eq:ub}
% \begin{aligned}
%     & \norm{\sum_{i=1}^{n}\grd \ell(\prm_i^t; Z_i^{t+1})}^2 = \norm{ \sum_{i=1}^{n} \big[ \grd \ell(\prm_i^t; Z_i^{t+1}) - \grd f_i( \prm_i^t; \prm_i^t ) + \grd f_i( \prm_i^t; \prm_i^t ) - \grd f_i( \thps; \thps)  \big] }^2 \\
%     &\leq 2 n \sum_{i=1}^{n} \EE_t \norm{ \grd \ell(\prm_i^t; Z_i^{t+1}) - \grd \ell(\thps; Z_{i}^{t+1})}^2 + 2 \sum_{i=1}^{n} \EE_t \norm{\grd \ell(\thps; Z_i^{t+1}) - \grd f_i(\thps; \thps)}^2
%     \\
%     &\leq 2n \sum_{i=1}^{n}L^2\norm{\prm_i^t - \thps}^2 
%     + 2 \sum_{i=1}^{n} \sigma^2 (1+\norm{\prm_i^t - \thps}^2)
%     \\
%     &\leq 2(nL^2 + \sigma^2)\sum_{i=1}^{n}\norm{\prm_i^t-\thps}^2 + 2n\sigma^2 {\color{red}
%     \leq 4(nL^2 + \sigma^2)\sum_{i=1}^{n} \left\{ \norm{\prm_i^t-\Bprm^t}^2 + \norm{ \Tprm^t }^2 \right\} + 2n\sigma^2
%     }
% \end{aligned}
% \end{eqnarray}
% where we  in the first inequality. To derive the third inequality, we apply the assumption  A\ref{ass: smooth} and A\ref{ass:SecOrdMom}.

Next, we focus on the inner product term in \eqref{rec1}, we have
\beq 
\begin{aligned}
\Pscal{ \Tprm^t }{ \sum_{i=1}^{n} \grd f_i(\prm_i^t, \prm_i^t) } & = \sum_{i=1}^{n} \Pscal{\Tprm^t}{\grd f_i(\prm_i^t; \prm_i^t)-\grd f_i(\Bprm^t; \thps)} \\
&\quad + \sum_{i=1}^{n} \Pscal{\Tprm^t }{\grd f_i(\Bprm^t; \thps)-\grd f_i(\thps; \thps)} 
\end{aligned}
\eeq 
Applying the Cauchy-Schwarz inequality and A\ref{ass: smooth}, A\ref{ass:sensitive}, we obtain 
\begin{align} 
& \sum_{i=1}^{n} \Pscal{\Tprm^t}{\grd f_i(\prm_i^t; \prm_i^t)-\grd f_i(\Bprm^t; \thps)} \geq -\norm{\Tprm^t } \sum_{i=1}^{n}\left( L\norm{\prm_i^t-\Bprm^t}+L\varepsilon_i \norm{\prm_i^t-\thps}\right) \notag \\
& \geq - \norm{\Tprm^t } \sum_{i=1}^{n}\left( L (1 + \epsilon_i) \norm{\prm_i^t-\Bprm^t}+ L \epsilon_i \norm{\Tprm^t} \right) . \label{eq:lem3key2}
% = - n L \epsilon_{\sf avg} \norm{\Tprm^t}^2 - L (1+\epsilon_{\sf max}) \sum_{i=1}^n \norm{\Tprm^t} \norm{ \prm_i^t - \Bprm^t }
\end{align}
Meanwhile, using the strong convexity property of $\ell(\cdot; \cdot)$ [cf.~A\ref{ass: strongcvx}], we have
\beq 
\sum_{i=1}^{n} \Pscal{\Tprm^t }{\grd f_i(\Bprm^t; \thps)-\grd f_i(\thps; \thps)} \geq n\mu \norm{ \Tprm^t }^2.
\eeq
Summing up the two lower bounds and rearranging terms give
\begin{align}
\frac{1}{n} \EE_t \Pscal{ \Tprm^t }{ \sum_{i=1}^{n}\grd f_i(\prm_i^t, \prm_i^t) } 
& \geq (\mu - L \epsilon_{\sf avg} ) \norm{\Tprm^t }^2 - \frac{L}{n} (1+\epsilon_{\sf max} ) \sum_{i=1}^{n} \norm{\Tprm^t }\norm{\prm_i^t - \Bprm^t} .
\end{align}

% Since $\norm{\prm_i^t-\thps}\leq \norm{\prm_i^t-\Bprm^t} + \norm{\Bprm^t -\thps}$, above upper bound can be further simplified as following
% \begin{align*}
%     \frac{\gamma_{t+1}}{n}\EE_t \Pscal{\Bprm^t-\thps}{\sum_{i=1}^{n}\grd  \ell(\prm_i^t; Z_i^{t+1})} &\geq
%     \frac{\gamma_{t+1}}{n}\Bigg(
%     n[\mu-(\sigma + L\varepsilon)] \norm{\Bprm^t-\thps}^2 - 
%     [\sigma +L(1+\varepsilon)] \sum_{i=1}^{n}\norm{\Bprm^t-\thps}\cdot \norm{\prm_i^t-\Bprm^t}\Bigg)
% \end{align*}
For any $\alpha > 0$, using the Young's inequality shows that the above can be further lower bounded by
\beq \label{eq:lb}
\begin{aligned}
% & \left[ \mu - L\epsilon_{\sf avg} - \frac{\alpha}{2n} L(1+\epsilon_{\sf max}) \right] \norm{\Tprm^t }^2 - \frac{ L }{2n\alpha} \sum_{i=1}^{n}(1+\epsilon_i)\norm{\prm_i^t-\Bprm^t}^2 
& \left[ \mu - L\epsilon_{\sf avg} - \frac{\alpha}{2n} L(1+\epsilon_{\sf max}) \right] \norm{\Tprm^t }^2 - \frac{ L(1+\epsilon_{max}) }{2n\alpha} \sum_{i=1}^{n}\norm{\prm_i^t-\Bprm^t}^2
\\
& \geq \left[\mu-L\epsilon_{avg}-\frac{\alpha}{2n}L(1+\epsilon_{\sf max}) \right] \norm{ \Tprm^t }^2 - \frac{L(1+\epsilon_{\sf max})}{2n\alpha} \norm{\CSE{t}}_F^2  \\
& \geq \left[\mu-(1+\delta)L\epsilon_{\sf avg} \right] \norm{ \Tprm^t }^2 - \frac{L(1+\epsilon_{max})^2 }{ 4n^2 \delta \epsilon_{\sf avg} } \norm{\CSE{t}}_F^2,  
\end{aligned}
\eeq 
where we have set $\alpha=\frac{2n \delta \epsilon_{\sf avg}}{1+\epsilon_{\sf max}}$ to yield the last inequality.
% Above lower bound can be rewritten as following

Substituting \eqref{eq:ub}, \eqref{eq:lb} back to the inequality \eqref{rec1} gives us the desired result. In particular,
\beq 
\begin{aligned}
\EE_t \norm{ \Tprm^{t+1} }^2 & \leq \norm{ \Tprm^t }^2 - 2 \gamma_{t+1} \left[ \left[\mu-(1+\delta) L \epsilon_{\sf avg} \right] \norm{ \Tprm^t }^2 - \frac{L(1+\epsilon_{max})^2 }{ 4n^2 \delta \epsilon_{\sf avg} } \norm{\CSE{t}}_F^2 \right] \\
& \quad + \gamma_{t+1}^2 \left[ \frac{ 2 \sigma^2 }{n} + c_2 \norm{ \Tprm^t }^2 + c_2 \frac{1}{n} \norm{ \CSE{t} }_F^2 \right] \\
& = \left( 1 - 2\tmu \gamma_{t+1} + c_2 \gamma_{t+1}^2 \right)\norm{\Tprm^t}^{2} + \left[ c_1 \frac{\gamma_{t+1}}{n} + c_2 \frac{\gamma_{t+1}^2}{n}\right] \norm{\CSE{t}}^2_F + \frac{2\sigma^2}{n}\gamma_{t+1}^2 \\
& \leq \left( 1 - \tmu \gamma_{t+1} \right)\norm{\Tprm^t}^{2} + \left[ c_1 \frac{\gamma_{t+1}}{n} + c_2 \frac{\gamma_{t+1}^2}{n}\right] \norm{\CSE{t}}^2_F + \frac{2\sigma^2}{n}\gamma_{t+1}^2
\end{aligned}
\eeq
where we recall the constants $c_1 \eqdef \frac{L(1+\epsilon_{\sf max})^2}{2{n}\delta \epsilon_{\sf avg}}$, $c_2 \eqdef 4 \left( \frac{ \sigma^2 }{n} + L^2 ( 1 + \epsilon_{\sf max} )^2 \right)$ and $\tmu \eqdef \mu - (1+\delta)\epsilon_{\sf avg} L$ and the last inequality is obtained by observing the condition $\gamma_{t+1} \leq \tmu / c_2$.

%%%%%%%%%%%%%%%%%%%%%%%%%
%%%%%%%%%%%%%%%%%%%

\section{\bf Proof of Lemma \ref{lem: consens}} \label{appendix:lem_consens}
To simplify notations, we denote
\beq \label{eq:grdtF_def}
\begin{aligned}
    \Tgrd F^t \eqdef \left( \grd \ell(\prm_1^t; Z_1^{t+1}) , \cdots, \grd \ell(\prm_n^t; Z_n^{t+1}) \right)^\top \in \RR^{n\times d},
    \\
    \Prm^t \eqdef \big( \prm_1^t, \cdots, \prm_n^t \big)^\top \in \RR^{n\times d},~ \BPrm^t \eqdef (1/n) {\bf 1}{\bf 1}^\top \Prm^t \in \RR^{n} .
\end{aligned}
\eeq 
Notice that $\CSE{t} = \Prm^t - \BPrm^t = ( {\bm I} - (1/n) {\bf 1}{\bf 1}^\top ) \Prm^t$.
% Note that A\ref{ass: graph} implies that there exists a matrix ${\bm U} \in \RR^{n \times (n-1)}$ with orthogonal columns such that 
% \[
% {\bm I}-\frac{1}{n}{\bm 1}{\bm 1}^\top = {\bm U} {\bm U}^\top 
% \]
% and $\norm{{\bm U}^\top {\bm W}{\bm U}}_2 \leq 1 - \rho$. 
We first observe the following relation:
\begin{align} %\label{eq:consensuserr}
\CSE{t+1} & = \Prm^{t+1}-\BPrm^{t+1} = \left({\bm I}-\frac{1}{n}{\bm 1}{\bm 1}^\top \right) \Prm^{t+1} = \left({\bm I}-\frac{1}{n}{\bm 1}{\bm 1}^\top \right) \left({\bm W}\Prm^t -\gamma_{t+1}\Tgrd F^t\right) \nonumber  \\
&= \left({\bm W}-\frac{1}{n}{\bm 1}{\bm 1}^\top \right) \CSE{t} - \gamma_{t+1} \left(  {\bm I} - \frac{1}{n} {\bf 1}{\bf 1}^\top \right) \Tgrd F^t, \nonumber
\end{align}
where the last equality is due to $( {\bm I} - (1/n) {\bf 1}{\bf 1}^\top ) {\bm W} = ( {\bm W} - (1/n) {\bf 1}{\bf 1}^\top ) ( {\bm I} - (1/n) {\bf 1}{\bf 1}^\top )$ as ${\bm W}$ is a doubly stochastic matrix.

Computing the squared norm of the consensus error leads to: for any $\alpha > 0$,
\begin{align}
\EE_t \norm{\CSE{t+1}}^2_F & \leq (1+\alpha)(1-\rho)^2 \norm{\CSE{t}}_F^2 + (1+\frac{1}{\alpha})\gamma_{t+1}^2 \EE_t \norm{ \left(  {\bm I} - \frac{1}{n} {\bf 1}{\bf 1}^\top \right) \Tgrd F^t}_F^2 \nonumber \\
& \leq (1-\rho)\norm{\CSE{t}}_F^2 + \frac{\gamma_{t+1}^2}{\rho} \EE_t \norm{ \left(  {\bm I} - \frac{1}{n} {\bf 1}{\bf 1}^\top \right) \Tgrd F^t}_F^2, \label{eq3}
\end{align}
where we have applied A\ref{ass: graph} in the first inequality and set $\alpha = \frac{\rho}{1-\rho}$ in the second inequality. 
The last term in the above inequality can be bounded as
\begin{eqnarray}\label{eq2}
\begin{aligned}
&\EE_t \norm{ \left( {\bm I} - \frac{1}{n} {\bf 1}{\bf 1}^\top \right) \Tgrd F^t}_F^2 = \EE_t \left[ \sum_{i=1}^{n}\norm{\grd \ell(\prm_i^t; Z_i^{t+1}) - \frac{1}{n}\sum_{j=1}^{n}\grd \ell(\prm_j^t; Z_j^{t+1})}^2 \right] \\
& \leq 3 \sum_{i=1}^{n}\EE_t\norm{\grd \ell(\prm_i^t; Z_i^{t+1})-\grd f_i(\prm_i^t, \prm_i^t ) }^2 + \frac{3}{n}  \sum_{j=1}^{n} \EE_t\norm{ \grd \ell(\prm_j^t; Z_j^{t+1}) - \grd f_j(\prm_j^t, \prm_j^t) }^2 \\
& \quad + 3 \sum_{i=1}^{n}\norm{\grd f_i(\prm_i^t, \prm_i^{t}) - \frac{1}{n}\sum_{j=1}^{n}\grd f_j(\prm_j^t, \prm_j^t)}^2 \\
& \leq 6 \sigma^2 \left(n+\sum_{i=1}^{n}\norm{\prm_i^t-\thps}^2\right)
+ 3 \sum_{i=1}^{n}\norm{\grd f_i(\prm_i^t, \prm_i^{t}) - \frac{1}{n}\sum_{j=1}^{n}\grd f_j(\prm_j^t, \prm_j^t)}^2 \\
& \leq 6 \sigma^2 \left(n  + 2n \norm{\Tprm^t}^2 + 2 \norm{\CSE{t}}_F^2 \right)
+ 3 \sum_{i=1}^{n}\norm{\grd f_i(\prm_i^t, \prm_i^{t}) - \frac{1}{n}\sum_{j=1}^{n}\grd f_j(\prm_j^t, \prm_j^t)}^2
\end{aligned}
\end{eqnarray}
where the second last inequality is due to A\ref{ass:SecOrdMom}. 
For each $i = 1, \ldots, n$, we observe 
\begin{align}
& \norm{\grd f_i(\prm_i^t,\prm_i^t) \!-\! \grd f_i(\Bprm^t, \Bprm^t) \!+\!\grd f_i(\Bprm^t, \Bprm^t)\!-\!\frac{1}{n}\sum_{j=1}^{n}\grd f_j(\Bprm^t, \Bprm^t) \!-\! \frac{1}{n}\sum_{j=1}^{n}[\grd f_j(\prm_j^t,\prm_j^t)\!-\!\grd f_j(\Bprm^t, \Bprm^t)]}^2 \notag \\
&\leq 3 \norm{\grd f_i(\prm_i^t,\prm_i^t) - \grd f_i(\Bprm^t, \Bprm^t)}^2 + 3 \norm{\grd f_i(\Bprm^t, \Bprm^t)-\frac{1}{n}\sum_{j=1}^{n}\grd f_j(\Bprm^t, \Bprm^t)}^2 \label{eq:noA6bp} \\
& \quad + \frac{3}{n}\sum_{j=1}^{n}\norm{\grd f_j(\prm_j^t,\prm_j^t)-\grd f_j(\Bprm^t, \Bprm^t)}^2 \notag \\
& \leq 3 \norm{\grd f_i(\prm_i^t,\prm_i^t) - \grd f_i(\Bprm^t, \Bprm^t)}^2 + \frac{3}{n}\sum_{j=1}^{n}\norm{\grd f_j(\prm_j^t,\prm_j^t)-\grd f_j(\Bprm^t, \Bprm^t)}^2 + 3 \varsigma^2 \left( 1 + \norm{ \Tprm^t }^2 \right) \notag
\end{align}
where the last inequality is due to A\ref{ass:hete}. Now, we observe 
\begin{align*}
& \sum_{i=1}^{n} \norm{\grd f_i(\prm_i^t, \prm_i^{t}) - \frac{1}{n}\sum_{j=1}^{n}\grd f_j(\prm_j^t, \prm_j^t)}^2 \leq 6 \sum_{i=1}^n \norm{\grd f_i(\prm_i^t,\prm_i^t) - \grd f_i(\Bprm^t, \Bprm^t)}^2 + 3n \varsigma^2 \left( 1 + \norm{ \Tprm^t }^2 \right) \\
& \leq 6L^2(1+ {\epsilon_{\sf max}} )^2 \norm{\CSE{t}}_F^2 + 3n \varsigma^2 \left(1+\norm{\Tprm^t}^2\right)
\end{align*}
where the second inequality is due to Lemma \ref{lem: grd_dev} and the definition of $\CSE{t}$. 

Substituting the above bounds into \eqref{eq2} leads to
\beq\label{eq:a}
\begin{aligned}
& \EE_t \norm{ \left( {\bm I} - \frac{1}{n} {\bf 1}{\bf 1}^\top \right) \Tgrd F^t}_F^2 \\
& \leq 6 \sigma^2 \left(n  + 2n \norm{\Tprm^t}^2 + 2 \norm{\CSE{t}}_F^2 \right) + 18 L^2(1+ {\epsilon_{\sf max}} )^2 \norm{\CSE{t}}_F^2 + 9 n \varsigma^2 \left( 1 + \norm{\Tprm^t}^2\right) \\
& \leq 9n [ \sigma^2 + \varsigma^2 ] + 12 n [ \sigma^2 + \varsigma^2 ] \norm{\Tprm^t}^2 + \left[ 12 \sigma^2 + 18L^2(1+ \epsilon_{\sf max})^2 \right] \norm{ \CSE{t} }_F^2 
\end{aligned}
\eeq 
% \begin{align*}
% \EE_t \norm{ \left( {\bm I} - \frac{1}{n} {\bf 1}{\bf 1}^\top \right) \Tgrd F^t}_F^2 
%     &\leq 3\Bigg\{ 2\sigma^2\Big(n+\sum_{j=1}^{n}\norm{\prm_j^t-\thps}^2\Big) + 6L^2(1+\epsilon_{max})^2\sum_{j=1}^{n}\norm{\prm_i^t-\thps}^2
%     \\
%     &\quad + 3n\varsigma^2 \left(1+\norm{\Bprm^t-\thps}^2\right)
%     \Bigg\}
%     \\
%     &\leq 6\left[\sigma^2 + 3L^2(1+\aaa{\epsilon_{max}})^2 \right]\sum_{j=1}^{n}\norm{\prm_j^t-\thps}^2 + 9n(\sigma^2+\varsigma^2) + 9n\varsigma^2 \norm{\Bprm^t-\thps}^2
% \end{align*}
% Using Cauchy-Shwarz inequality to deal with $\norm{\prm_j^t-\thps}^2$, 
% \begin{align*}
%     \EE_t\left[\norm{{\bm U}{\bm U}^\top \Tgrd F^t}_F^2\right] &\leq 12\left[\sigma^2 + 3L^2(1+\aaa{\epsilon_{max}})^2 \right]\sum_{j=1}^{n}\norm{\prm_j^t-\Bprm^t}^2 
%     \\
%     &\quad + \left( 12\left[\sigma^2 + 3L^2(1+\aaa{\epsilon_{max}})^2 \right] + 9\varsigma^2 \right)n\norm{\Bprm^t-\thps}^2 + 9n(\sigma^2+\varsigma^2)
% \end{align*}
Let $c_3 := 12 \sigma^2 + 18L^2(1+ \epsilon_{\sf max})^2$. 
Substituting the above inequality into \eqref{eq3} gives us
\begin{align*}
\EE_t \norm{\CSE{t+1}}^2_F & \leq 
(1-\rho)\norm{\CSE{t}}_F^2 + \frac{ \gamma_{t+1}^2 }{\rho } \left( 12n [ \sigma^2 + \varsigma^2 ]  \norm{\Tprm^t}^2 + c_3 \norm{ \CSE{t} }_F^2 \right)  
+ 9n(\sigma^2+\varsigma^2)\frac{\gamma_{t+1}^2}{\rho} \\
& \leq 
(1-\rho/2)\norm{\CSE{t}}_F^2 + \frac{ \gamma_{t+1}^2 }{\rho } 12n [ \sigma^2 + \varsigma^2 ] \norm{\Tprm^t}^2 
+ 9n(\sigma^2+\varsigma^2)\frac{\gamma_{t+1}^2}{\rho},
\end{align*}
where the last inequality is due to the step size condition $\gamma_{t+1}^2 \leq \rho^2 / 2c_3$.
% and requiring $\gamma_{t+1}^2 \leq \rho^2 / \big(24\left[\sigma^2+3L^2(1+\epsilon_{max})^2\right]\big)$ give us
% \begin{align*}
%     \EE_t \norm{\CSE{t+1}}^2_F \leq (1-\rho/2)\norm{\CSE{t}}_F^2 + \frac{n\gamma_{t+1}^2}{\rho}\norm{\Bprm^t-\thps}^2\left(c_3 + 9\varsigma^2 \right)
%     +9n(\sigma^2+\varsigma^2)\frac{\gamma_{t+1}^2}{\rho}
% \end{align*}
% where $c_3\eqdef 12[\sigma^2 + 3L^2(1+\epsilon_{max})^2]$. Also, it can be rewritten as 
% \begin{align*}
%     \EE_t[n^{-1}\norm{\CSE{t+1}}^2_F] \leq (1-\rho/2)n^{-1}\norm{\CSE{t}}_F^2 + \frac{\gamma_{t+1}^2}{\rho}\norm{\Bprm^t-\thps}^2\left(c_3 + 9\varsigma^2 \right)
%     +9(\sigma^2+\varsigma^2)\frac{\gamma_{t+1}^2}{\rho}
% \end{align*}
The proof is concluded. 

\paragraph{Alternative Bound without A\ref{ass:hete}}
We consider bounding \eqref{eq:noA6bp} without using A\ref{ass:hete}. Instead, we only assume that $\max_{i=1,\ldots,n} \norm{ \grd f_i( \thps; \thps ) }^2 \leq \varsigma^2$. We observe 
\beq 
\begin{aligned}
& \norm{\grd f_i(\Bprm^t, \Bprm^t)- \grd f(\Bprm^t, \Bprm^t)}^2 \leq 2 \norm{ \grd f_i ( \thps; \thps ) }^2 \\
& \qquad \qquad + 2 \norm{ \grd f_i(\Bprm^t, \Bprm^t) - \grd f_i( \thps; \thps ) + \grd f( \thps; \thps ) - \grd f( \Bprm^t; \Bprm^t ) }^2 \\
& \leq 2 \norm{ \grd f_i ( \thps; \thps ) }^2 + 8 L^2 \left( 1 + \epsilon_{\sf max} \right)^2 \norm{ \Tprm^t }^2
\leq 2 \varsigma^2 + 8 L^2 \left( 1 + \epsilon_{\sf max} \right)^2 \norm{ \Tprm^t }^2,
\end{aligned}
\eeq 
for all $i=1,\ldots,n$.
This leads to
\begin{align*}
& \textstyle \sum_{i=1}^{n} \norm{\grd f_i(\prm_i^t, \prm_i^{t}) - \frac{1}{n}\sum_{j=1}^{n}\grd f_j(\prm_j^t, \prm_j^t)}^2 \\
& \leq 6L^2(1+ {\epsilon_{\sf max}} )^2 \norm{\CSE{t}}_F^2 + 2 n \varsigma^2 + 8 n L^2 \left( 1 + \epsilon_{\sf max} \right)^2 \norm{ \Tprm^t }^2.
\end{align*}
Subsequently, 
\beq 
\begin{aligned}
& \textstyle \EE_t \norm{ \left( {\bm I} - \frac{1}{n} {\bf 1}{\bf 1}^\top \right) \Tgrd F^t}_F^2 \\
& \leq 6 \sigma^2 \left(n  + 2n \norm{\Tprm^t}^2 + 2 \norm{\CSE{t}}_F^2 \right) + 6 n \varsigma^2 + 6L^2 (1+\epsilon_{\sf max})^2 \left( 3 \norm{\CSE{t}}_F^2 + 4 n \norm{ \Tprm^t }^2 \right) \\
& = 6 n [ \sigma^2 + \varsigma^2 ] + 12n \left[ \sigma^2 + 2 L^2 \left( 1 + \epsilon_{\sf max} \right)^2 \right] \norm{\Tprm^t}^2 + \left[ 12 \sigma^2 + 18L^2(1+ \epsilon_{\sf max})^2 \right] \norm{ \CSE{t} }_F^2 
\end{aligned}
\eeq
Taking $c_3 := 12 \sigma^2 + 18L^2(1+ \epsilon_{\sf max})^2$ as before and substituting the inequality into \eqref{eq3} yields
\begin{align}
& \EE_t \norm{\CSE{t+1}}^2_F \notag \\
& \leq 
(1-\rho)\norm{\CSE{t}}_F^2 + \frac{ \gamma_{t+1}^2 }{\rho } \left( 12n \left[ \sigma^2 + 2 L^2 \left( 1 + \epsilon_{\sf max} \right)^2 \right]  \norm{\Tprm^t}^2 + c_3 \norm{ \CSE{t} }_F^2 \right)  
+ 6n(\sigma^2+\varsigma^2)\frac{\gamma_{t+1}^2}{\rho} \notag \\
& \leq 
(1-\rho/2)\norm{\CSE{t}}_F^2 + \frac{ \gamma_{t+1}^2 }{\rho } 12n \left[ \sigma^2 + 2 L^2 \left( 1 + \epsilon_{\sf max} \right)^2 \right] \norm{\Tprm^t}^2 + 6n(\sigma^2+\varsigma^2)\frac{\gamma_{t+1}^2}{\rho}, \notag 
\end{align}
where the last inequality is due to $\sup_{t \geq 1} \gamma_t \leq \rho / \sqrt{2 c_3}$.
The above can be simplified into 
\beq \label{eq:cons_noA6}
{\textstyle \frac{1}{n}} \EE_t \norm{\CSE{t+1}}^2_F \leq 
\left(1- \frac{\rho}{2} \right) {\textstyle \frac{1}{n}} \norm{\CSE{t}}_F^2 + \frac{ \gamma_{t+1}^2 }{\rho } 12 \left[ \sigma^2 + 2 L^2 \left( 1 + \epsilon_{\sf max} \right)^2 \right] \norm{\Tprm^t}^2 + 6(\sigma^2+\varsigma^2)\frac{\gamma_{t+1}^2}{\rho}.
\eeq 
Compared to \eqref{lem:consens_eq}, we observe that the above bound entails a larger coefficient for $\normtxt{\Tprm^t}^2$ which lead to a (slightly) worse convergence bound for the {\bname} scheme.

Lastly, we should mention that as in the original Lemma~\ref{lem: consens}, \eqref{eq:cons_noA6} can also be combined with Lemma~\ref{lem:descent} to develop an alternate version of Lemma~\ref{lem:lya}. Subsequently, we can achieve a similar result as Theorem~\ref{thm1} without assuming A\ref{ass:hete}.

\allowdisplaybreaks
%%%%%%%%%%%%%%%%%%%%
\section{Proof of Lemma~\ref{lem:lya}}\label{app:lem_lya}
Combining Lemmas~\ref{lem:descent} and \ref{lem: consens} leads to
% \beq
\begin{align*} 
{\cal L}_{t+1} & \leq ( 1 - \tmu \gamma_{t+1}) \, \EE \norm{\Tprm^t}^{2} + \left[c_1 \gamma_{t+1} + c_2 \gamma_{t+1}^2 \right] {\textstyle \frac{1}{n}} \EE \norm{\CSE{t}}^2_F + \frac{2\sigma^2}{n}\gamma_{t+1}^2 \\
& \quad + \gamma_{t+1}  \frac{ 8 c_1 }{ \rho } \left( \left(1- \frac{\rho}{2} \right) {\textstyle \frac{1}{n}} \EE \norm{\CSE{t}}_F^2 + \frac{ \gamma_{t+1}^2 }{\rho } 12 [ \sigma^2 + \varsigma^2 ] \EE \norm{\Tprm^t}^2 
+ 9 (\sigma^2+\varsigma^2)\frac{\gamma_{t+1}^2}{\rho} \right) \\
& = \left( 1 - \tmu \gamma_{t+1} + \frac{96 c_1}{\rho^2} [\sigma^2 + \varsigma^2] \gamma_{t+1}^3 \right) \EE \norm{ \Tprm^t }^2 + \frac{2\sigma^2}{n}\gamma_{t+1}^2 + \frac{ 72 c_1 }{ \rho^2 } (\sigma^2+\varsigma^2) \gamma_{t+1}^3 \\
& \quad + \gamma_t \frac{8 c_1}{\rho} \left( \frac{ \gamma_{t+1} }{ \gamma_t } \left( 1 - \frac{\rho}{2} \right) + \frac{\rho}{8} + \frac{c_2 \rho}{8 c_1} \gamma_{t+1} \right) {\textstyle \frac{1}{n}} \EE \norm{\CSE{t}}_F^2
\end{align*} 
% \eeq
Note that by the step size conditions specified in the lemma, we have 
\beq 
\begin{split}
& 1 - \tmu \gamma_{t+1} + \frac{96 c_1}{\rho^2} [\sigma^2 + \varsigma^2] \gamma_{t+1}^3 \leq 1 - \tmu \gamma_{t+1} / 2 \\
& \frac{ \gamma_{t+1} }{ \gamma_t } \left( 1 - \frac{\rho}{2} \right) + \frac{\rho}{8} + \frac{c_2 \rho}{8 c_1} \gamma_{t+1} \leq 1 - \tmu \gamma_{t+1} / 2.
\end{split} 
\eeq 
Thus, we obtain
\beq \label{eq:lem_lya_part1}
\begin{aligned}
{\cal L}_{t+1} & \leq (1 - \tmu \gamma_{t+1}/2) {\cal L}_t 
+ \frac{2\sigma^2}{n}\gamma_{t+1}^2 + \frac{ 72 c_1 }{ \rho^2 } (\sigma^2+\varsigma^2) \gamma_{t+1}^3 .
% & \leq \prod_{s=1}^{t+1} ( 1 - \tmu \gamma_s/2 ) {\sf D} + \frac{8 \sigma^2}{ \tmu n} \gamma_{t+1} + \frac{ 288 c_1 }{ \rho^2 \tmu } ( \sigma^2 + \varsigma^2 ) \gamma_{t+1}^2.
\end{aligned}
\eeq 
This concludes the first part of the lemma, i.e., \eqref{eq:lem_lya_part0}. For the second part, we further expand \eqref{eq:lem_lya_part1} to obtain 
\beq 
\begin{aligned}
{\cal L}_{t+1} & \leq \prod_{i=1}^{t+1} \left( 1 - \frac{\tmu \gamma_i}{ 2 } \right) {\sf D} + \sum_{s=1}^{t+1} \prod_{i=s+1}^{t+1} (1 - \tmu \gamma_i/2) \left( \frac{2\sigma^2}{n}\gamma_{s}^2 + \frac{ 72 c_1 }{ \rho^2 } (\sigma^2+\varsigma^2) \gamma_{s}^3 \right) \\
& \leq \prod_{i=1}^{t+1} \left( 1 - \frac{\tmu \gamma_i}{ 2 } \right) {\sf D} + \frac{ 288 c_1 ( \sigma^2 + \varsigma^2 ) }{ \rho^2 \tmu } \gamma_{t+1}^2 + \frac{ 8 \sigma^2 }{ \tmu n } \gamma_{t+1}.
\end{aligned}
\eeq 
where we recall that ${\sf D} := \normtxt{\Tprm^0}^2 + \frac{8 \gamma_1 c_1}{\rho n} \norm{ \CSE{0} }_F^2$ and the last inequality is due to Lemma~\ref{lem:aux} together with the specified step size conditions. 
The proof is thus concluded.

%%%%%%%%%%%%%%%%%%%%%%%%%%%
\section{Auxilliary Results}
\begin{lemma}\label{lem:aux}
Consider a sequence of non-negative, non-increasing step sizes $\{\gamma_{t}\}_{t \geq 1}$. Let $a>0$, $p\in \ZZ_+$ and $\gamma_{1}<2 / a$. If $\gamma_{t}^p / \gamma_{t+1}^p \leq 1+(a / 2) \gamma_{t+1}^p$ for any $t \geq 1$, then
\beq 
\sum_{j=1}^{t} \gamma_{j}^{p+1} \prod_{\ell=j+1}^{t}\left(1-\gamma_{\ell} a\right) \leq \frac{2}{a} \gamma_{t}^p,~~\forall~t \geq 1.
\eeq 
\end{lemma}

% {\color{red} where is the proof?}
\begin{proof}
Observe that:
\begin{align*}
\sum_{j=1}^{t} \gamma_{j}^{p+1} \prod_{\ell=j+1}^{t}\left(1-\gamma_{\ell} a\right) & =\gamma_{t}^p \sum_{j=1}^{t} \gamma_{j} \prod_{\ell=j+1}^{t} \frac{\gamma_{\ell-1}^p}{\gamma_{\ell}^p}\left(1-\gamma_{\ell} a\right) \\
& \overset{(a)}{\leq} \gamma_{t}^p \sum_{j=1}^{t} \gamma_{j} \prod_{\ell=j+1}^{t}\left(1-\gamma_{\ell} \frac{a}{2} \right) \\
&=\frac{2 \gamma_{t}^p}{a} \sum_{j=1}^{t} \left( \prod_{\ell=j+1}^{t} \left(1-\gamma_{\ell} a / 2\right) - \prod_{\ell^{\prime}=j}^{t} \left(1-\gamma_{\ell^{\prime}} a / 2\right)\right) \\
& = \frac{2 \gamma_{t}^p}{a} \left( 1 - \prod_{\ell^{\prime}=1}^{t} \left( 1-\gamma_{\ell^{\prime}} a / 2 \right) \right) \leq \frac{2 \gamma_{t}^p}{a},
\end{align*}
where (a) is due to the following observation
\[
\frac{\gamma_{\ell-1}^p}{\gamma_{\ell}^p}\left(1-\gamma_{\ell} a\right) 
\leq \left(1 + \frac{a}{2} \gamma_\ell^p \right) (1 - \gamma_\ell a ) \leq 1 - \frac{a}{2} \gamma_\ell .
\]
The proof is concluded.
\end{proof}

\begin{lemma}\label{lem:aux2}
Consider a sequence of non-negative, non-increasing step sizes $\{\gamma_{t}\}_{t \geq 1}$. Let $p \in \ZZ^+$. If $\sup_{t \geq 1}\gamma_{t}^p / \gamma_{t+1}^p \leq  1+ \frac{\rho}{4-2\rho}$, then for any $t \geq 0$, it holds that
\beq 
    \sum_{i=1}^{t+1}\left(1-\frac{\rho}{2} \right)^{t+1-i} \gamma_{i}^p \leq \frac{4}{\rho}\gamma_{t+1}^p.
\eeq 
\end{lemma}
\begin{proof}
We observe the following chain:
\begin{align*}
\sum_{i=1}^{t+1}\left(1-\frac{\rho}{2} \right)^{t+1-i} \gamma_{i}^p 
& = \gamma_{t+1}^p \sum_{i=1}^{t+1} \left(1-\frac{\rho}{2} \right)^{t+1-i} \left( \frac{\gamma_{i}}{\gamma_{i+1}} \right)^p \left( \frac{\gamma_{i+1}}{\gamma_{i+2}} \right)^p \cdots \left( \frac{\gamma_{t}}{\gamma_{t+1}} \right)^p \\
&\leq \gamma_{t+1}^p\sum_{i=1}^{t+1} (1-\frac{\rho}{4})^{t+1-i} 
\leq \frac{4}{\rho} \gamma_{t+1}^p
\end{align*}
where the second last inequality is due to:
\[
    (1-\rho/2)\left( \frac{\gamma_{i+1}}{\gamma_{i+2}}\right)^p \leq 1-\frac{\rho}{4} 
\]
since $\sup_{k\geq 1}\gamma_{k-1}^p / \gamma_{k}^p \leq 1+ \frac{\rho}{4-2\rho}$.
This completes the proof. 
\end{proof}

\section{Details of Numerical Experiments and Additional Results}
This section provides details for the numerical experiments conducted in \S\ref{sec:num}. We also describe an additional numerical experiment based on the logistic regression Example~\ref{example1} on synthetic data. The latter examines the effects of heterogeneous data on the convergence rate of {\bname}.
% which highlights the effects of 

For all our experiments, we have performed {\bname} with the step size $\gamma_t = a_0 / (a_1 + t)$. Moreover, at each iteration of {\bname}, the $i$th agent draws ${\tt batch} \geq 1$ samples from ${\cal D}_i( \prm_i^t )$. The parameters $a_0 > 0, a_1 \geq 0, {\tt batch} \geq 1$ used for different tasks are specified in Table~\ref{tab:step}. For both Gaussian mean estimation and {\tt spambase} logistic regression, we use the same parameters for all settings of $\bar{\epsilon} = \epsilon_{\sf avg}$.
We consider using $n=25$ agents in all experiments, connected on a ring graph. We set the mixing matrix weights as $W_{ij} = 1/3$ for all $(i,j) \in E$, and $W_{ij} = 0$ if $(i,j) \notin E$.\vspace{.1cm}

\begin{table}[t]
\begin{center}
\begin{tabular}{ l l l l l } 
 \toprule 
 \textbf{Tasks} & $\bar{\epsilon} = \epsilon_{\sf avg}$ & $a_0$ & $a_1$ & {\tt batch} \\ 
  \midrule
 Gaussian Mean Estimation & see \S\ref{sec:num} & 50 & 10000 & 1 \\ 
%   \hline
 Spam Email Classification & see \S\ref{sec:num} & 50 & 100000 & 32 \\ 
%   \hline
  {\tt LEAF} Synthetic Data (Hetero \& Homo) & 0.1 & 200 & 1000 & 32 \\
%   \hline
  {\tt LEAF} Synthetic Data (Hetero \& Homo) & 10.0 & 1 & 1000 & 32 \\
  \bottomrule
\end{tabular}
\end{center}
\caption{Parameters for the numerical experiments.} \label{tab:step} 
\end{table}

% \textbf{Gaussian Mean Estimation}

\textbf{Spam Email Classification.} In Fig.~\ref{fig:email_spam_app}, we provide additional results for the experiment in the main paper [cf.~Fig.~\ref{fig:email_spam}]. In particular, we compare the training loss $f( \Bprm^t ; \Bprm^t )$ and training accuracy against the iteration number $t$. We also plot the  gap to an \emph{approximate} Multi-PS solution in Fig.~\ref{fig:email_spam_app} (right) for $\| \Bprm^t - \hat{\thps} \|^2$. Note that the Multi-PS solution compared here is only an approximation obtained by applying a similar method to repeated gradient descent in \citep{perdomo2020performative} on $\min_{\prm} \sum_{i=1}^n f_i ( \prm ; \prm )$, where we used 1000 gradient descent iterations together with an outer loop of $10^4$ deployments. Note that this process is only guaranteed to find a near-optimal solution, denoted as $\hat{\thps}$. Nevertheless, we observe that when the decision dependent distributions becomes more sensitive ($\epsilon_{\sf avg} = 1$), the {\bname} scheme seems unable to reach $\hat{\thps}$.\vspace{.1cm}

\begin{figure}
    \centering
    \includegraphics[width=.985\textwidth]{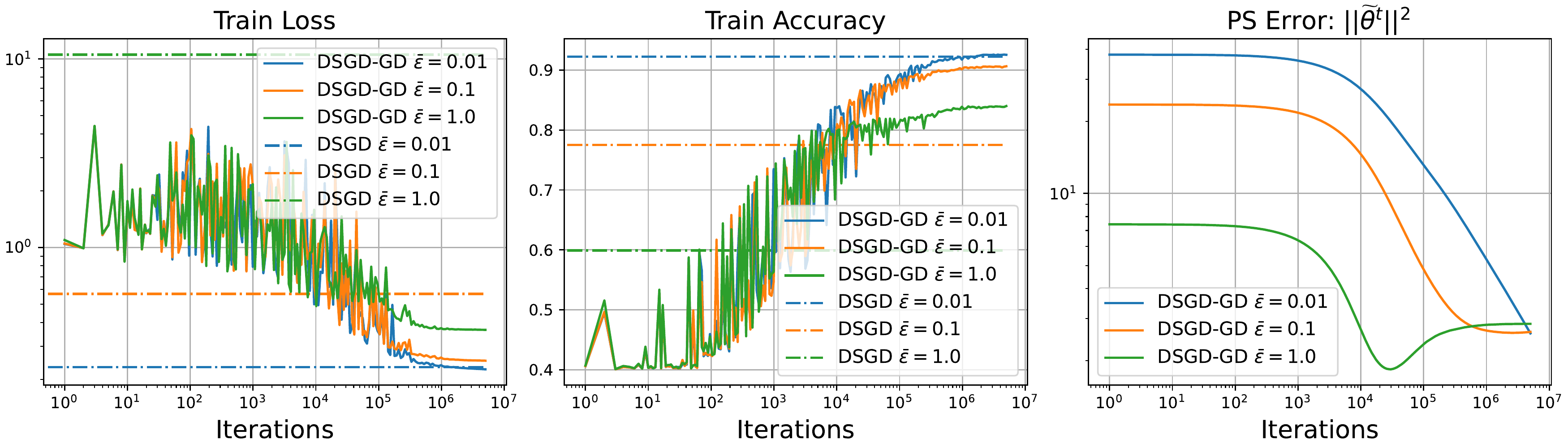}\vspace{-.2cm}
    \caption{\textbf{Additional Results for Spam Email Classification.} 
    (Left) Training Loss. (Middle) Training Accuracy. (Right) Approximate Gap to Multi-PS solution (see below). We also compare the non-performative optimal solution (dashed lines) on the shifted dataset.}\vspace{-.3cm}
    \label{fig:email_spam_app}
\end{figure}

% \textbf{Accuracy on the testing dataset}

\textbf{Logistic Regression on {\tt LEAF} Synthetic Data.}
To study the effect of homogeneity of data distribution [cf.~A\ref{ass:hete}] on the convergence of {\bname}, we conduct an additional experiment based on Example~\ref{example1} but on the {\tt LEAF} synthetic data \citep{caldas2019leaf}. Here, we set the sensitivity parameter at $\epsilon_i = \bar{\epsilon}$ for $i = 1, \ldots, 25$ and generate synthetic data using the framework in \citep{caldas2019leaf} with the standard deviation $\sigma = 1$ that represents the degree of heterogeneity of the dataset. 
Note that the framework produces $m_i = 100$ training samples with $d=100$ features for each agent, denoted as $({\bm X}_k^i, Y_k^i)_{k=1}^{100}$, for $i=1,\ldots,25$ agents. 

We consider two settings and describe them using the notations as in Example~\ref{example1}. In the {\tt heterogeneous} data setting, the base data distribution ${\cal D}_i^0$ for agent $i$ is taken to be $({\bm X}_k^i, Y_k^i)_{k=1}^{100}$ such that ${\cal D}_i ( \prm ) \neq {\cal D}_j( \prm )$. In the {\tt homogeneous} data setting, the base data distribution ${\cal D}_i^0$ for agent $i$ is taken to be $( ({\bm X}_k^i, Y_k^i)_{k=1}^{100} )_{i=1}^{25}$, i.e., the entire dataset generated from {\tt LEAF}. Note that in this case, ${\cal D}_i^0 \equiv {\cal D}_j^0$ and thus ${\cal D}_i ( \prm ) \equiv {\cal D}_j( \prm )$ for any $\prm \in \RR^d$ and $i,j = 1,\ldots, n$ since $\epsilon_i = \epsilon_{\sf avg}$. Note that the Multi-PS solution $\thps$ (if exists) in both settings are unique and identical. Meanwhile, the {\tt homogeneous} case satisfies A\ref{ass:hete} with $\varsigma = 0$, thus the {\bname} scheme applied to it is expected to converge at a faster rate than in the {\tt heterogeneous} case. 

Our numerical results are presented in Fig.~\ref{fig:leaf_app}, and we show in Table~\ref{tab:step} the simulation parameters. Observe that with $\epsilon_{\sf avg} = 10$, the local data distributions are too sensitive and the Multi-PS solution $\thps$ may not exist. With $\epsilon_{\sf avg} = 0.1$, we observe that the convergence of test accuracy, training loss, etc.~are faster with the {\tt homogeneous} case initially. However, as the iteration number $t$ grows, the gap between the {\tt homogeneous} and {\tt heterogeneous} cases fade. This corroborates with our finite-time analysis in \eqref{eq:net_ind}, where the fluctuation term $\sigma^2 \gamma_t / (n \tmu)$ becomes dominant as $t \gg 1$ in all cases, yet the transient time can be shorter when $\varsigma = 0$ as predicted by \eqref{eq:transient}.

% We investigate the effect of homogeneity of data distribution over the convergence of DSGD-GD. We adopt the same setup in {\bf Spam Email Classification} except that we choose $\epsilon_i = \bar{\epsilon}~\forall i$. Each agent has $m_i = 100$ training samples of $d=100$ features.
\begin{figure}
    \centering
    \includegraphics[width=.985\textwidth]{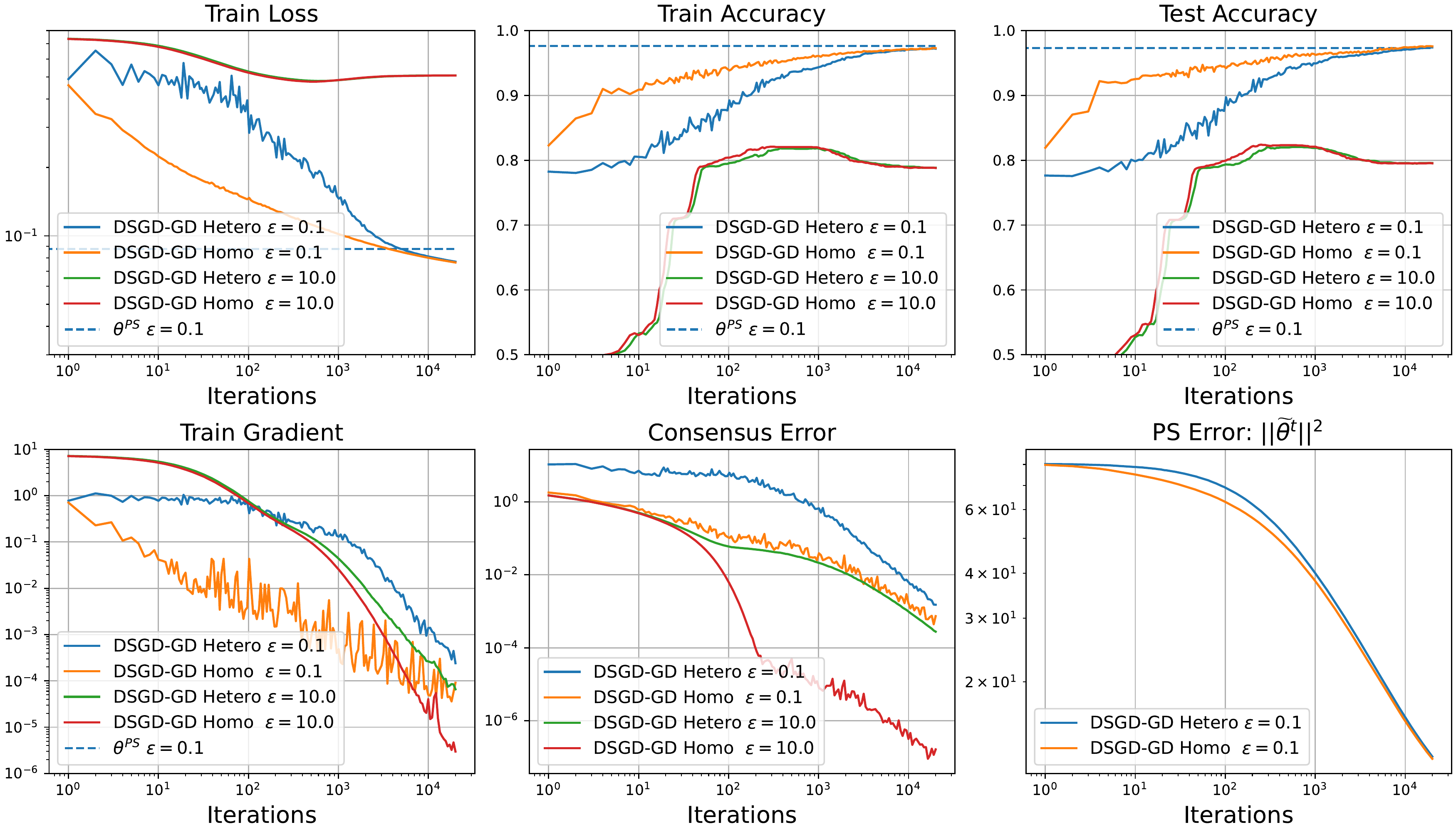}\vspace{-.2cm}
    \caption{\textbf{Logistic Regression on {\tt LEAF} Synthetic Data.}
    DSGD-GD in homogeneous and heterogeneous data distribution converge to the same Multi-PS solution.\vspace{-.3cm}
    }
    \label{fig:leaf_app}
\end{figure}

\newpage 
%%%%%%%%%%%%% Extension  %%%%%%%%%%%%%%%%%
\section{Extension to Time-varying Graph}\label{sec:timevarying}

This section shows how to extend our analysis for {\bname} to the setting with time-varying communication graph. Let $G^{(t)} = (V, E^{(t)})$ be a simple, undirected graph which is possibly not connected and the graph is associated with a weighted adjacency matrix ${\bm W}^{(t)}$. Note that the graph $G^{(t)}$ consists of a fixed set of agents $V$ and a set of time-varying edges $E^{(t)}$. 

In lieu of A\ref{ass: graph}, we assume that:
\begin{Assumption}\label{ass: graph2}
The time-varying undirected graph sequence $\{G^{(t)}\}_{t \geq 1} = \{ (V,E^{(t)}) \}_{t \geq 1}$ is $B$-connected. Specifically, for any $t \geq 1$, there exists a positive integer $B$ such that the undirected graph $(V, E^{(t)} \cup \cdots E^{(t+B-1)})$ is connected. For any $t\geq 1$, the mixing matrix $\bm{W}^{(t)}\in \RR^{n \times n}$ satisfies:
\begin{enumerate}[leftmargin=*, itemsep=.1mm, topsep=.1mm]
    \item (Topology) ${\bm W}_{ij}^{(t)} = 0$ if $(i,j) \notin E^{(t)}$.
    \item (Doubly stochastic) ${\bm W}^{(t)}\mathbf{1}=({\bm W}^{(t)})^\top\mathbf{1}=\mathbf{1}$.
    \item (Fast mixing) Let ${\bm A}^{(t)}\eqdef {\bm W}^{(t)}-\frac{1}{n}{\mathbf 1}\mathbf{ 1}^\top$, there exists $\brho \in (0,1]$ such that $\left\| {\bm A}^{(t+B-1)}\cdots {\bm A}^{(t)}  \right\|_{2} \leq 1-{\brho}$.
\end{enumerate}

The last condition can be guaranteed under the bounded communication setting, i.e., when the combined graph $(V, E^{(t)}\cup\cdots E^{(t+B)})$ is connected for any $t\geq 0$.
\end{Assumption}

\textbf{Notations.} Throughout, we denote $\Prm(m,n) \eqdef \EE[ \norm{\CSE{m}}_F^2 + \cdots + \norm{\CSE{n}}_F^2 ]$ and $\Tprm(m, n) \eqdef \EE[  \normtxt{\Tprm^m}^2  + \cdots + \normtxt{\Tprm^n}^2 ]$, which is the aggregation of consensus error and performative stable gap in one time block whose length is $B$, respectively.\vspace{.2cm}

\textbf{Proof Sketch.} Below we provide a proof sketch for the convergence of {\bname} scheme when the latter is applied on a time varying graph satisfying A\ref{ass: graph2}.
We begin by considering the extensions of Lemmas \ref{lem:descent} and \ref{lem: consens}. As follows,

\begin{lemma}[\bf Extension of Lemma \ref{lem:descent}]\label{lem:descent2} Fix any $\delta > 0$ and let $\epsilon_{\sf avg}\leq \frac{\mu}{(1+\delta) L}$. Under A\ref{ass: strongcvx}, A\ref{ass: smooth}, A\ref{ass:sensitive}, A\ref{ass:SecOrdMom} and let the step sizes satisfy $\sup_{t \geq 0} \gamma_{t+1} \leq \frac{ \tmu }{ c_2 }$, the following bound holds for any $t \geq 0$,
\begin{align*}
    \Tprm(t+1, t+B) &\leq (1-\tmu \gamma_{t+B})^B \Tprm(t-B+1, t) + \frac{2B\sigma^2}{n}\gamma_{t+1}^2 
    \\
    &\quad
    + {B}\Big(c_1\frac{\gamma_{t+1}}{n} 
     + c_2\frac{\gamma_{t+1}^2}{n} \Big)\left[\Prm(t-B+1, t) + \Prm(t+1, t+B)\right].
\end{align*}
\end{lemma}
\begin{proof} Recall the inequality (\ref{lem:des_eq}) in Lemma \ref{lem:descent},
\begin{align}\label{eq:des}
    \EE_t \norm{\Tprm^{t+1}}^2 \leq ( 1 - \tmu \gamma_{t+1})\norm{\Tprm^t}^{2} + \left[c_1 \gamma_{t+1} + c_2 \gamma_{t+1}^2 \right] {\textstyle \frac{1}{n}} \norm{\CSE{t}}^2_F + \frac{2\sigma^2}{n}\gamma_{t+1}^2.
\end{align}
This implies
\begin{align*}
    \Tprm(t+1, t+B)\leq (1-\tmu \gamma_{t+B})\Tprm(t, t+B-1) + \left( \frac{c_1\gamma_{t+1}}{n} +  \frac{c_2\gamma_{t+1}^2}{n}\right) \Prm(t, t+B-1) + \frac{2B\sigma^2}{n}\gamma_{t+1}^2,
\end{align*}
where we have summed \eqref{eq:des} from $t+1$th to $t+B$th iteration and noted that the step size $\gamma_{t}$ is non-increasing. Applying the above inequality for $B$ times, we can link two consecutive $B$ performative stable gap $\Tprm(t+1, t+B)$ and $\Tprm(t-B+1, t)$ by
\begin{eqnarray}\label{eq:part}
\begin{aligned}
    &\Tprm(t+1, t+B) \leq (1-\tmu \gamma_{t+B})^B\Tprm(t-B+1, t) + \frac{2B\sigma^2}{n} \gamma_{t+1}^2 
    \\
    &\quad + \left( \frac{c_1\gamma_{t+1}}{n} +  \frac{c_2\gamma_{t+1}^2}{n}\right) \left[ \Prm(t, t+B-1) + \Prm(t-1, t+B-2) +\cdots + \Prm(t-B+1, t)\right].
\end{aligned}
\end{eqnarray}
For the first term $\Prm(t, t+B-1)$ in the last quantity, we observe the crude bound
\[
    \Prm(t, t+B-1) \leq \Prm(t+1, t+B) + \EE \normtxt{\CSE{t}}_F^2
    \leq \Prm(t-B+1, t) + \Prm(t+1, t+B).
\]
Following the same trick, we get another crude bound as
\[
\begin{aligned}
& \left[ \Prm(t, t+B-1) + \Prm(t-1, t+B-2) +\cdots + \Prm(t-B+1, t)\right] \\
& \leq B[\Prm(t-B+1, t) + \Prm(t+1, t+B)].
\end{aligned}
\]
Substituting back to inequality (\ref{eq:part}) derives the final bound
\begin{align*}
    \Tprm(t+1, t+B) &\leq (1-\tmu \gamma_{t+B})^B \Tprm(t-B+1, t) + \frac{2B\sigma^2}{n}\gamma_{t+1}^2 
    \\
    &\quad
    + {B}\Big(c_1\frac{\gamma_{t+1}}{n} 
     + c_2\frac{\gamma_{t+1}^2}{n} \Big)\left[\Prm(t-B+1, t) + \Prm(t+1, t+B)\right].
\end{align*}
\end{proof}
% We can further extend Lemma \ref{lem: consens} on the block consensus error.
%%%%%%%%%%%%%%%%
\begin{lemma}[\bf Extension of Lemma~\ref{lem: consens}]\label{lem: consens2}
Under A\ref{ass: smooth}--A\ref{ass:SecOrdMom} and A\ref{ass: graph2} and let the step sizes satisfy 
\[
    \sup_{t \geq 0} \gamma_{t+1} \leq {\rho}/\sqrt{2 B c_3},
\]
then it holds for any $t \geq 0$ that
\begin{eqnarray}
    \begin{aligned}
        \Prm(t+1, t+B) &\leq \frac{1-\brho/2}{1-B c_3 \gamma_{t-B+1}^2/\brho} \Prm(t-B+1,t) 
        \\
        &\quad + \frac{\gamma_{t-B+1}^2}{\rho - B c_3 \gamma_{t-B+1}^2} \left\{ B^2 d_1 + d_2 B [\Tprm(t-B+1, t) + \Tprm(t+1, t+B)]\right\},
    \end{aligned}
\end{eqnarray}
where $d_1 \eqdef 9n(\sigma^2 + \varsigma^2)$, $d_2 \eqdef 12n(\sigma^2 + \varsigma^2)$.
\end{lemma}
\begin{proof}
Recall the notations \eqref{eq:grdtF_def} and observe that
\begin{align*}
    \CSE{t+1} = \Prm^{t+1} - \BPrm^{t+1} = \underbrace{\left( {\bm W}^{t+1} - \frac{1}{n}{\bm 1}{\bm 1}^\top \right)}_{={\bm A}^{t+1}}\CSE{t} - \gamma_{t+1}\left({\bm I} - \frac{1}{n}{\bm 1}{\bm 1}^\top\right)\Tgrd F^t  .
\end{align*}
Therefore, we can obtain the following consensus error recursion 
\[
    \CSE{t+1} = {\bm A}^{t+1} \CSE{t} - \gamma_{t+1} \left({\bm I} - (1/n){\bm 1}{\bm 1}^\top\right) \Tgrd F^t.
\]  
Then, we aim to link $\CSE{t+1}$ to $\CSE{t-B+1}$.
\begin{align*}
    \CSE{t+1}&= {\bm A}^{t+1}  \CSE{t} -\gamma_{t} \left({\bm I} - (1/n){\bm 1}{\bm 1}^\top\right) \Tgrd F^{t-1}
    \\
    &= {\bm A}^{t+1} {\bm A}^t \CSE{t-1} -\gamma_{t} {\bm A}^{t+1} \left({\bm I} - (1/n){\bm 1}{\bm 1}^\top\right) \Tgrd F^{t-1} - \gamma_{t+1} \left({\bm I} - (1/n){\bm 1}{\bm 1}^\top\right) \Tgrd F^{t}
    \\
    & \quad \vdots
    \\
    &= {\bm A}^{t+1} {\bm A}^{t} {\bm A}^{t-1}\cdots {\bm A}^{t-B+1}\CSE{t-B+1} - \sum_{s=t-B+1}^{t}\gamma_{s+1} {\bm A}^{s+2} \left({\bm I} - (1/n){\bm 1}{\bm 1}^\top\right) \Tgrd F^s.
\end{align*}
Taking Frobenius norm on both sides and applying the Young's inequality give
\[
\begin{split}
    \norm{\CSE{t+1}}_F^2 & \leq (1+\alpha) \norm{{\bm A}^{t+1} {\bm A}^{t} {\bm A}^{t-1}\cdots {\bm A}^{t-B+1}}^2 \norm{\CSE{t-B+1}}_F^2 \\
    & + (1 + \alpha^{-1}) \sum_{s=t-B+1}^{t}\gamma_{s+1}^2 \norm{{\bm A}^{s+2}}^2 \norm{ \left({\bm I} - (1/n){\bm 1}{\bm 1}^\top\right) \Tgrd F^s}_F^2,
\end{split}
\]
which holds for any $\alpha > 0$.
Using A\ref{ass: graph2} and setting $\alpha = \frac{\rho}{1-\rho}$, we have
\[
    \norm{\CSE{t+1}}_F^2\leq (1-\brho) \norm{\CSE{t-B+1}}_F^2 + \sum_{s=t-B+1}^{t}\gamma_{s+1}^2\norm{ \left({\bm I} - (1/n){\bm 1}{\bm 1}^\top\right) \Tgrd F^s}_F^2.
\]
Similarly, we get 
\begin{align*}
    \norm{\CSE{t+2}}_F^2 &\leq (1-\brho) \norm{\CSE{t-B+2}}_F^2 + \sum_{s=t-B+2}^{t+1}\gamma_{s+1}^2\norm{ \left({\bm I} - (1/n){\bm 1}{\bm 1}^\top\right) \Tgrd F^s}_F^2
    \\
    &\vdots
    \\
    \norm{\CSE{t+B}}_F^2 &\leq (1-\brho) \norm{\CSE{t}}_F^2 + \sum_{s=t}^{t+B-1}\gamma_{s+1}^2\norm{ \left({\bm I} - (1/n){\bm 1}{\bm 1}^\top\right) \Tgrd F^s}_F^2.
\end{align*}

% \begin{align*}
%     \CSE{t+2} &= A^{t+2} A^{t+1} A^{t}\cdots A^{t-B+3}\CSE{t-B+2} - \sum_{s=t-B+2}^{t+1}\gamma_{s+1} A^{s+2} {\bm U}{\bm U}^\top \Tgrd F^s
%     \\
%     &\cdots
%     \\
%     \CSE{t+B} &= A^{t+B} A^{t+B-1} \cdots A^{t+1}\CSE{t} - \sum_{s=t}^{t+B}\gamma_{s+1} A^{s+2} {\bm U}{\bm U}^\top \Tgrd F^s
% \end{align*}

% \begin{align*}
%     \Prm(t+1, t+B) &\leq (1-\brho)\Prm(t-B+1, t) 
%     \\
%     &+ \frac{1}{\rho} \left\{
%     \norm{\sum_{s=t-B+1}^{t}\gamma_{s+1} A^{s+2} {\bm U}{\bm U}^\top \Tgrd F^s}^2  + \cdots +  \norm{\sum_{s=t}^{t+B}\gamma_{s+1} A^{s+2} {\bm U}{\bm U}^\top \Tgrd F^s}^2 \right\}
% \end{align*}
% where we set $\alpha = \frac{\rho}{1-\rho}$. Noticing that $\norm{A^{t}}\leq 1$.
Adding these $B$ consensus errors together leads to
\beq\label{eq:b}
\begin{aligned}
    & \Prm(t+1, t+B) \leq (1-\brho)\Prm(t-B+1, t) 
    \\
    &+ \frac{\gamma_{t-B+1}^2 }{\rho} \left\{
    \sum_{s=t-B+1}^{t} \norm{ \left({\bm I} - (1/n){\bm 1}{\bm 1}^\top\right) \Tgrd F^s}_F^2  + \cdots +  \sum_{s=t}^{t+B}\norm{ \left({\bm I} - (1/n){\bm 1}{\bm 1}^\top\right) \Tgrd F^s}_F^2 \right\}.
\end{aligned}
\eeq
Using the inequality (\ref{eq:a}) in the proof of Lemma \ref{lem: consens}, we get
\begin{align*}
    \EE_{s} \norm{ \left({\bm I} - (1/n){\bm 1}{\bm 1}^\top\right) \Tgrd F^s}_F^2 & \leq d_1 + d_2 \norm{\Tprm^s}^2 + c_3 \norm{ \CSE{s} }_F^2,
\end{align*}
where $d_1 \eqdef 9n(\sigma^2 + \varsigma^2)$, $d_2 \eqdef 12n(\sigma^2 + \varsigma^2)$ and $c_3 = 12 \sigma^2 + 18L^2(1+ \epsilon_{\sf max})^2$. Then, we have
\begin{align*}
    \sum_{s=t-B+1}^{t}\EE\norm{ \left({\bm I} - (1/n){\bm 1}{\bm 1}^\top\right) \Tgrd F^s}_F^2 &\leq \sum_{s=t-B+1}^{t} \EE\left[ d_1 +d_2 \normtxt{\Tprm^s}^2 + c_3 \norm{\CSE{s}}_F^2 \right]
    \\
    & = B d_1 + d_2 \sum_{s=t-B+1}^{t} \EE \normtxt{\Tprm^t}^2 + c_3 \sum_{s=t-B+1}^{t} \EE \norm{\CSE{s}}_F^2.
\end{align*}
Substituting back to (\ref{eq:b}) give us
\begin{align*}
    \Prm(t+1, t+B) &\leq (1-\brho) \Prm(t-B+1, t) +
    \frac{\gamma_{t-B+1}^2}{\rho} \Bigg\{ B^2 d_1 + d_2 [\Tprm(t-B+1, t) + \cdots + \Tprm(t,t+B)] 
    \\
    &\quad + c_3 [\Prm(t-B+1, t)+\cdots +\Prm(t, t+B)] \Bigg\}.
\end{align*}
The above can be simplified to
\begin{align*}
    \Prm(t+1, t+B) &\leq (1-\brho) \Prm(t-B+1, t) +
    \frac{\gamma_{t-B+1}^2}{\rho} \Big\{ B^2 d_1 + d_2 B [\Tprm(t+1, t+B) + \Tprm(t,t+B)] 
    \\
    &\quad + c_3 B[\Prm(t-B+1, t) + \Prm(t+1, t+B)]\Big\}.
\end{align*}
Setting $\sup_{k\geq 1}\gamma_{k} \leq \frac{\brho}{\sqrt{2 c_3 B}}$ and rearranging terms give us
\begin{align*}
    \Prm(t+1, t+B) &\leq \frac{1-\brho/2}{1-B c_3 \gamma_{t-B+1}^2/\brho} \Prm(t-B+1,t) 
    \\
    &\quad + \frac{\gamma_{t-B+1}^2}{\rho - B c_3 \gamma_{t-B+1}^2} \left\{ B^2 d_1 + d_2 B [\Tprm(t-B+1, t) + \Tprm(t+1, t+B)]\right\},
\end{align*}
which gives us desired upper bound for $\Prm(t+1, t+B)$.
\end{proof}

\textbf{Convergence of $\Bprm^t$ to $\thps$ with Time varying graph.}
We conclude our proof sketch through analyzing the following Lyapunov function. For any $t \geq 0$, we define:
\[
    {\cal L}_{t+1}^{t+B} \eqdef \Tprm(t+1, t+B) + \Prm(t+1, t+B) \geq 0.
\]
Combing Lemma \ref{lem:descent2} and \ref{lem: consens2} leads to
\beq \label{eq:lyap_B}
\begin{aligned}
    &\left(1-B c_1\frac{\gamma_{t+1}}{n} - B c_2 \frac{\gamma_{t+1}^2}{n}\right)\Prm(t+1, t+B) + \left(1-\frac{d_2 B \gamma_{t-B+1}^2}{\rho - B c_3 \gamma_{t-B+1}^2}\right)\Tprm(t+1, t+B) 
    \\
    &\leq \left( \frac{1-\brho/2}{1-B c_3 \gamma_{t-B+1}^2 / \brho} + B c_1\frac{\gamma_{t+1}}{n} + B c_2 \frac{\gamma_{t+1}^2}{n} \right) \Prm(t-B+1, t) 
    \\
    &\quad + \left( (1-\tmu \gamma_{t+B})^B + \frac{d_2 B \gamma_{t-B+1}^2}{\rho - B c_3 \gamma_{t-B+1}^2} \right)\Tprm(t-B+1, t) + \frac{B^2 d_1 \gamma_{t-B+1}^2}{\rho - B c_3 \gamma_{t-B+1}^2} + \frac{2B \sigma^2}{n}\gamma_{t+1}^2 .
\end{aligned}
\eeq 
We focus on the l.h.s.~of above inequality. If the step size satisfies 
\[
    \sup_{k\geq 1} \gamma_{k} \leq \min \left\{ \frac{c_1}{c_2}, \sqrt{\frac{\brho}{2B c_3}}, \frac{\brho c_1}{n}\right\}
\]
then, the l.h.s.~of \eqref{eq:lyap_B} can be lower bounded by
\[
\text{l.h.s. of \eqref{eq:lyap_B}} \geq \left( 1-2B c_1 \frac{\gamma_{t+1}}{n}\right)\left[ \Prm(t+1, t+B) + \Tprm(t+1, t+B) \right] .
\]
% Next, we consider the r.h.s.~of \eqref{eq:lyap_B}. Suppose that {\color{red} $\sup_{k\geq 1}\gamma_{k} \leq \sqrt{\frac{\rho}{ (4/\brho - 1) B c_3}}$}, it holds
Next, we consider the r.h.s.~of \eqref{eq:lyap_B}. Suppose that { $\sup_{k\geq 1}\gamma_{k} \leq \sqrt{\frac{\rho}{ (4 - \brho) B c_3}}$}, it holds
\[
    \frac{1-\brho/2}{1-B c_3 \gamma_{t-B+1}^2 / \brho} \leq 1-\brho/4, \qquad  \frac{B^2 d_1 \gamma_{t-B+1}^2}{\rho - B c_3 \gamma_{t-B+1}^2} \leq \frac{2B^2 d_1}{\rho} \gamma_{t-B+1}^2.
\]
If step size also satisfies 
\[
    \sup_{k\geq 1}\gamma_{k} \leq \min\left\{ \frac{1}{\sqrt{b}}, \frac{\tmu \brho}{2^{2B+1} d_2 B} \right\},
\]
where $b$ such that $\gamma_{k}^2/\gamma_{k+1}^2 \leq 1 + b\gamma_{k+1}^2$, then it holds:
\begin{align*}
    & \text{r.h.s.~of \eqref{eq:lyap_B}} \\
    & \leq \left( 1-\frac{\brho}{4} + 2B c_1 \frac{\gamma_{t+1}}{n} \right) \Prm(t-B+1, t)
     + \left( 1-\frac{\tmu \gamma_{t+B}}{2} \right)\Tprm(t-B+1, t) \\
    & + \frac{2 B^2 d_1}{\rho} \gamma_{t-B+1}^2 + \frac{2B \sigma^2}{n}\gamma_{t+1}^2.
\end{align*}
Combining the above inequalities lead to:
\begin{align*}
    &\left( 1-2B c_1 \frac{\gamma_{t+1}}{n}\right)\left[ \Prm(t+1, t+B) + \Tprm(t+1, t+B) \right] \leq {\left( 1 - \frac{\brho}{4} + 2B c_1 \frac{\gamma_{t+1}}{n} \right)} \Prm(t-B+1, t) 
    \\
    &+ \left( 1-\frac{\tmu \gamma_{t+B}}{2} \right)\Tprm(t-B+1, t) + \frac{2 B^2 d_1}{\rho} \gamma_{t-B+1}^2 + \frac{2B \sigma^2}{n}\gamma_{t+1}^2.
\end{align*}
If the step size satisfying
\[
    \sup_{k\geq 1}\gamma_{k} \leq \frac{\brho}{8 B c_1 /n+ 2\tmu},
\]
then the main recursion can be simplified as
\begin{align*}
    &\left( 1-2B c_1 \frac{\gamma_{t+1}}{n}\right)\left[ \Prm(t+1, t+B) + \Tprm(t+1, t+B) \right] 
    \\
    &\leq   \left( 1-\frac{\tmu \gamma_{t+B}}{2} \right) \left[ \Prm(t-B+1, t) + \Tprm(t-B+1, t) \right]
     + \frac{2 B^2 d_1}{\rho} \gamma_{t-B+1}^2 + \frac{2B \sigma^2}{n}\gamma_{t+1}^2.
\end{align*}
Dividing $\left( 1-2B c_1 \frac{\gamma_{t+1}}{n}\right)$ for the both sides, we obtain that
\begin{align*}
    {\cal L}_{t+1}^{t+B} &\leq   \frac{\left( 1-\tmu \gamma_{t+B}/2 \right)}{1-2 B c_1 \gamma_{t+1}/n} {\cal L}_{t-B+1}^{t} + \left(\frac{2 B^2 d_1}{\rho}  + \frac{2B^2 \sigma^2}{n} \right) \frac{\gamma_{t-B+1}^2 }{ \left( 1-2 B c_1 \gamma_{t+1}/n \right)}.
\end{align*}
Observe that with sufficiently small step size, the above recursion can be simplified to give a similar form as \eqref{eq:lem_lya_part0}. Solving the recursion then lead to ${\cal L}_{t+1}^{t+B} = {\cal O}(\gamma_{t-B+1})$ and the convergence of $\Tprm( t+1,t+B ) \to 0$.

Lastly, we remark that the above analysis only gives a crude bound to the convergence of {\bname} in the time varying graph setting. It is possible to give tighter bounds through further optimizing the constants in the above analysis.

{
% \color{blue}
\section{Extension to Local Distributions Influenced by All Agents}
This section outlines how to extend our analysis to the scenario when the local distributions ${\cal D}_i(\cdot)$ are simultaneously influenced by other agents in the network similar to the competitive \aname~considered by \citep{narang2022multiplayer, piliouras2022multi}.

We define the concatenated decision vector $\varprm := ( \prm_1 , \ldots, \prm_n ) \in \RR^{nd}$ and state the modified consensus \aname~problem \eqref{eq:multipfd} as follows
\beq \label{eq:multipfd2} \textstyle
\min_{ \prm_i \in \RR^d, \, i=1,\ldots,n }~\frac{1}{n} \sum_{i=1}^n \EE_{ Z_i \sim {\cal D}_i( {\color{red}\varprm} ) } \big[ \ell( \prm_i ; Z_i ) \big] ~~\text{s.t.}~~ \prm_i = \prm_j,~\forall~(i,j) \in E. 
\eeq 
With a slight abuse of notation, we also define $f_i( \prm; {\color{red} \varprm} ) := \EE_{ Z_i \sim {\cal D}_i( {\color{red}\varprm} ) } \big[ \ell( \prm_i ; Z_i ) \big]$.

Following \citep{narang2022multiplayer}, we consider the following modification to A\ref{ass:sensitive}: 
\begin{Assumption}\label{ass:sensitive2}
For any $i=1,\ldots,n$, there exists a constant $\epsilon_i>0$ such that
\beq 
{\cal W}_{1}({\cal D}_i(\varprm), {\cal D}_i(\varprm^\prime)) \leq \epsilon_i \norm{\varprm-\varprm^\prime},~\forall~\varprm^\prime, \varprm \in \RR^{nd},
\eeq
where ${\cal W}_{1} ( {\cal D}, {\cal D}' )$ denotes the Wasserstein-1 distance between the distributions ${\cal D}, {\cal D}'$.
\end{Assumption}
Specifically, we notice that if $\varprm$ satisfies the consensus constraint, i.e., $\varprm = {\bf 1}_n \otimes \prm = ( \prm , \ldots, \prm )$, then A\ref{ass:sensitive2} is equivalent to A\ref{ass:sensitive} with the latter's sensitivity parameter given by ${\epsilon}_i' = \sqrt{n} \epsilon_i$  since $\| {\bf 1}_n \otimes \prm - {\bf 1}_n \otimes \prm' \| = \sqrt{n} \| \prm - \prm' \|$. 
This observation immediately leads to the following corollary of Proposition~\ref{lem:exist}:
\begin{Corollary}\label{lem:exist2} Under A\ref{ass: strongcvx}, A\ref{ass: smooth}, A\ref{ass:sensitive2}. Define the map ${\cal M}: \RR^d \to \RR^d$
\beq \label{eq:map_M2} \textstyle 
    \mathcal{M}(\prm) = \argmin_{\prm^\prime \in \RR^d} \frac{1}{n} \sum_{i=1}^{n} f_{i}(\prm^\prime; {\bf 1}_n \otimes \prm) 
\eeq
If ${\color{red}\sqrt{n}} \epsilon_{\sf avg} < \mu / L$, then the map ${\cal M}(\prm)$ is a contraction with the unique fixed point $\thps = {\cal M}( \thps )$. 
If ${\color{red}\sqrt{n}} \epsilon_{\sf avg} \geq \mu / L$, then there exists an instance of \eqref{eq:map_M} where $\lim_{T \to \infty} \norm{ {\cal M}^T(\prm) } = \infty$.
\end{Corollary}
The proof is attained by simply observing that if $\prm_i = \prm_j$ (as constrained by \eqref{eq:map_M2} (and \eqref{eq:multipfd2}), then A\ref{ass:sensitive2} is equivalent to A\ref{ass:sensitive} with ${\epsilon}_i' = \sqrt{n} \epsilon_i$.

\textbf{Comparison to \citep{narang2022multiplayer}.} Notice that in \citep{narang2022multiplayer}, the existence of a performative stable equilibrium requires $\sqrt{ \sum_{i=1}^n \epsilon_i^2 } < \mu / L$. Meanwhile, Corollary~\ref{lem:exist2} requires $(1/\sqrt{n}) \sum_{i=1}^n \epsilon_i < \mu / L$. Due to norm equivalence, we have 
\[ \textstyle 
(1/\sqrt{n}) \sum_{i=1}^n \epsilon_i \leq \sum_{i=1}^n \epsilon_i^2.
\] 
Thus, the consensus constrained performative stable solution in cooperative \aname~will be attainable under a more relaxed condition than the competitive \aname.  

\textbf{{\tt DSGD-GD} Algorithm for \eqref{eq:multipfd2}.} The extension of Theorem~\ref{thm1} to \eqref{eq:multipfd2} via the {\tt DSGD-GD} algorithm is more involved and thus the details are skipped in this brief discussion. However, it remains straightforward to extend the analysis through a careful modification of Lemma~\ref{lem:descent} with A\ref{ass:sensitive2}. In particular, one only needs to pay attention to the use of A\ref{ass:sensitive2} in \eqref{eq:lem3key1} and \eqref{eq:lem3key2} for the proof. 

\textbf{Remarks.} Lastly, we emphasize that as explained in the main paper, the original scenario considered by \eqref{eq:multipfd} and A\ref{ass:sensitive} is relevant to the decentralized learning scenario of the current paper. It captures the effects of `geographical' barriers where the population of users are not simultaneously influenced by all agents. 
Nevertheless, a future direction is to study the \aname~problem (cooperative or competitive) where users can be influenced by the decisions from a \emph{few} neighboring agents. 
}

\end{document}